\documentclass[11pt,a4paper,reqno, draft]{amsart}
\usepackage{amsmath, amssymb, latexsym, enumerate,  graphicx,tikz, stmaryrd, eufrak,enumitem,MnSymbol}

\usepackage{thmtools}

\usepackage[pdftex]{hyperref}
\usepackage{cleveref}
\usepackage{bbm}
\usepackage{cases}
\usepackage{array}
\usepackage{tikz-cd}
\usepackage[all]{xy}

\usepackage[hmarginratio={1:1},vmarginratio={1:1},lmargin=80.0pt,tmargin=90.0pt]{geometry}

\usepackage{tikz}
\tikzset{anchorbase/.style={baseline={([yshift=-0.5ex]current bounding box.center)}}}
\usetikzlibrary{decorations.markings}
\usetikzlibrary{decorations.pathreplacing}
\usetikzlibrary{arrows,shapes,positioning,backgrounds}
\tikzstyle directed=[postaction={decorate,decoration={markings,
    mark=at position #1 with {\arrow{>}}}}]
\tikzstyle rdirected=[postaction={decorate,decoration={markings,
    mark=at position #1 with {\arrow{<}}}}]
    
\usepackage{graphicx}
\usepackage[vcentermath]{youngtab}
\usepackage{nicefrac}

 \newlength{\baseunit}               
 \newcount{\numlines}                
\newtheorem{theorem}[subsubsection]{Theorem}
\newtheorem{lemma}[theorem]{Lemma}
\newtheorem{prop}[theorem]{Proposition}
\newtheorem{corollary}[subsubsection]{Corollary}

\theoremstyle{definition}
\newtheorem{definition}[subsubsection]{Definition}

\newtheorem{remark}[theorem]{Remark}

\newtheorem{example}[subsubsection]{Example}

\newtheorem{thmA}{Theorem}

\newcommand{\PSh}{\mathsf{PSh}}
\newcommand{\Sh}{\mathsf{Sh}}
\newcommand{\HSh}{\mathsf{HSh}}
\newcommand{\fp}{\mathsf{fp}}
\newcommand{\Ind}{\mathsf{Ind}}

\newcommand{\free}{\mbox{-}{\mathsf{free}}}

\newcommand{\tto}{\twoheadrightarrow}

\newcommand{\Tens}{\mathcal{T}\hspace{-.5mm}ens}
\newcommand{\RMon}{\mathcal{R}\mathcal{M}\hspace{-.3mm}on}

\newcommand{\Ab}{\mathsf{Ab}}

\newcommand{\Top}{\mathcal{T}\hspace{-1mm}op}
\newcommand{\HTop}{\mathcal{H}\mathcal{T}\hspace{-1mm}op}
\newcommand{\Sw}{\mathcal{H}\vspace{-0.1mm}\mathcal{K}}

\newcommand{\Noy}{\mathsf{Noy}}

\newcommand{\charr}{\mathrm{char}}
\newcommand{\Mot}{\mathsf{Mot}}
\newcommand{\Dim}{\mathrm{Dim}}
\newcommand{\Sym}{\mathrm{Sym}}

\newcommand{\ABt}{\mathcal{A}\mathcal{B}_3}
\newcommand{\ABf}{\mathcal{A}\mathcal{B}_5}
\newcommand{\Cat}{\mathcal{C}at}

\newcommand{\Fr}{\mathrm{Fr}}

\newcommand{\Spec}{\mathrm{Spec}}

\newcommand{\Mod}{\mbox{-}\mathsf{Mod}}

\newcommand{\Rep}{\mathsf{Rep}}

\newcommand{\ba}{\mathbf{a}}
\newcommand{\bb}{\mathbf{b}}
\newcommand{\bc}{\mathbf{c}}

\newcommand{\bt}{\mathbf{t}}
\newcommand{\bT}{\mathbf{T}}
\newcommand{\bU}{\mathbf{U}}

\newcommand{\bV}{\mathbf{V}}

\newcommand{\bC}{\mathbf{C}}
\newcommand{\bB}{\mathbf{B}}
\newcommand{\bA}{\mathbf{A}}

\newcommand{\cX}{\mathcal{X}}
\newcommand{\cY}{\mathcal{Y}}

\newcommand{\cS}{\mathcal{S}}

\newcommand{\cT}{\mathcal{T}}
\newcommand{\HT}{\mathcal{H}\mathcal{T}}
\newcommand{\cR}{\mathcal{R}}

\newcommand{\cU}{\mathcal{U}}
\newcommand{\cM}{\mathcal{M}}
\newcommand{\cN}{\mathcal{N}}

\newcommand{\pre}{\mathrm{pre}}

\newcommand{\id}{\mathrm{id}}

\newcommand{\Vecc}{\mathsf{Vec}}
\newcommand{\coker}{\mathrm{coker}}
\newcommand{\Hom}{\mathrm{Hom}}
\newcommand{\End}{\mathrm{End}}
\newcommand{\Aut}{\mathrm{Aut}}
\newcommand{\ev}{\mathrm{ev}}
\newcommand{\co}{\mathrm{co}}
\newcommand{\Lan}{\mathrm{Lan}}
\newcommand{\Ob}{\mathrm{Ob}}

\newcommand{\im}{\mathrm{im}}
\newcommand{\op}{\mathrm{op}}

\newcommand{\OB}{\mathcal{OB}}
\newcommand{\EN}{\mathcal{EN}}
\newcommand{\MO}{\mathcal{MO}}

\newcommand{\Seq}{\mathcal{S}eq}
\newcommand{\mZ}{\mathbb{Z}}
\newcommand{\mQ}{\mathbb{Q}}
\newcommand{\mN}{\mathbb{N}}

\newcommand{\mC}{\mathbb{C}}
\newcommand{\mF}{\mathbb{F}}

\newcommand{\Yon}{\mathtt{Y}}
\newcommand{\tI}{\mathtt{I}}
\newcommand{\tS}{\mathtt{S}}
\newcommand{\tZ}{\mathtt{Z}}
\newcommand{\tHZ}{\mathtt{HZ}}
\newcommand{\tN}{\mathtt{N}}
\newcommand{\unit}{{\mathbf{1}}}
\newcommand{\Cone}{\mathrm{Cone}}

\newcommand{\Arr}{\mathrm{Arr}}

\begin{document}
\title[Homological kernels]{Homological kernels of monoidal functors}
\author{Kevin Coulembier}

\address{School of Mathematics and Statistics, University of Sydney, NSW 2006, Australia}
\email{kevin.coulembier@sydney.edu.au}



\begin{abstract}
We show that each rigid monoidal category $\ba$ over a field defines a family of universal tensor categories, which together classify all faithful monoidal functors from $\ba$ to tensor categories.
Each of the universal tensor categories classifies monoidal functors of a given `homological kernel' and can be realised as a sheaf category, not necessarily on $\ba$. This yields a theory of `local abelian envelopes' which completes the notion of monoidal abelian envelopes.
\end{abstract}

\maketitle

\section*{Introduction}

\subsection*{Weak abelian envelopes}\label{weakenv}

For a field $k$, let $\ba$ be a small $k$-linear rigid monoidal category in which the endomorphism algebra of the tensor unit is $k$ (a rigid monoidal category over~$k$). If such a category is also abelian, it is called a {\em tensor category over $k$}, see \cite{Del90, EGNO}.
The 2-category of tensor categories and exact linear monoidal functors is denoted by $\Tens$.

A {\em weak abelian envelope} of $\ba$ is a faithful $k$-linear monoidal functor $\ba\to\bT$ to a tensor category $\bT$ which induces equivalences
$$\Tens(\bT,\bT')\;\xrightarrow{\sim}\;[\ba,\bT']^\otimes_{faith},$$
for all tensor categories $\bT'$ over $k$, where the right-hand side is the category of faithful linear monoidal functors. If $\ba\to\bT$ is also full, it is an {\em abelian envelope}.

Examples of abelian envelopes have led to some spectacular applications in recent years, see for instance \cite{BEO, CEH, Deligne, EHS, Ostrik}, and they proved to be a great source for new interesting tensor categories. This has prompted the developed of extensive theory in \cite{BEO, AbEnv, PreTop, CEOP} and references therein. The most natural question in the subject asks when a given category $\ba$ admits a (weak) abelian envelope. We refer to \cite{CEOP} for some partial and conjectural answers. Many rigid monoidal categories $\ba$ over $k$, such as most tensor-triangulated ones, do not admit any faithful monoidal functors to tensor categories. Clearly these do not admit (weak) abelian envelopes. 

The current work establishes that as soon as there exist such faithful monoidal functors, then $\ba$ always defines universal tensor categories, namely a non-empty set of {\em local abelian envelopes}, which together play the role of an envelope. In particular, this set has cardinality 1 if and only if $\ba$ admits a weak abelian envelope. In other words, while the question of whether $\ba$ admits a weak abelian envelope is equivalent to the question of whether the 2-functor $[\ba,-]^\otimes_{faith}$ from $\Tens$ to the 2-category of groupoids is representable, we demonstrate that it is always `multi-representable'.  This also greatly extends the potential of abelian envelopes as a source of tensor categories.



\subsection*{Multi-representability of 2-functors} We adapt some terminology from \cite{Di1} to the 2-categorical setting in a `1-dimensional' way.
Let $\cM$ be a 2-category and $F$ a 2-functor $\cM\to\Cat$ to the 2-category of categories. For a family $\{F_\gamma,\gamma\in \Gamma\}$ of such 2-functors, we can define the 2-functor $\coprod_{\gamma\in\Gamma} F_\gamma$ which sends $X\in \cM$ to the category $\coprod_{\gamma} F_\gamma(X)$. Here, $\coprod_\gamma\bC_\gamma$ for a family of categories $\{\bC_\gamma\}$ stands for the obvious category with class of objects given by $\sqcup_\gamma\Ob\bC_\gamma$. 

A $2$-functor $F$ is {\em multi-representable} if there exists a family $\{F_\gamma,\gamma\in \Gamma(F)\}$ of representable 2-functors with $F\simeq \coprod_\gamma F_\gamma$. Clearly, such a family, when it exists, is uniquely defined. In particular, $F$ is then representable if and only if $|\Gamma(F)|=1$. The representing objects $X_\gamma\in\cM$ of $F_\gamma$ are the {\em locally representing objects of $F$}.


As an intermediary result, we will demonstrate the following example of the above situation. Let $\cM$ be the 2-category of AB5 abelian categories and exact faithful cocontinuous functors. For any small additive category $\ba$, we have a 2-functor $F_{\ba}:\cM\to\Cat$ which sends an AB5 category $\bC$ to the category of additive functors $\ba\to\bC$ (we ignore that this lifts to a 2-functor to the 2-category of preadditive or even abelian categories). We prove that this 2-functor is multi-representable; note that it is not representable unless $\ba$ is the zero category. The locally representing objects are Grothendieck categories, and a subset of them comprise the sheaf categories on $\ba$ for all additive Grothendieck topologies on $\ba$. Hence, the set $\Gamma(F_{\ba})$ is an extension of the set of topologies on $\ba$. We call this generalised notion {\em homological Grothendieck topologies}, and they can be interpreted in several ways as ordinary Grothendieck topologies, for instance on $K^b\ba$ or on the category obtained from $\ba$ by freely adjoining kernels.


An inclusion $U:\cM\to\cN$ of a 2-full sub-2-category is {\em multi-reflective} if $\cN(A,U-):\cM\to\Cat$ is multi-representable for every $A\in \cN$. Following \cite{Di2}, we can then define $\Spec_UA$ as the set $\Gamma(\cN(A,U-))$ labelling the locally representing objects of $\cN(A,U-)$.


\subsection*{Main result} We return to the original set up. Let $\RMon$ be the 2-category of rigid monoidal categories over $k$ and {\em faithful} linear monoidal functors. Denote by $U$ the 2-fully faithful forgetful 2-functor $\Tens\to\RMon$.
For $\theta\in\RMon(\ba,U\bT)$, we define its {\em homological kernel} (we will see various alternative guises of this kernel, in particular it contains the information of a homological topology) as the kernel of $K^b\ba(-,\unit)\to \bT(-,\unit)$, where the homological functor $K^b\ba\to\bT$ is the one induced from $\theta$ by taking $H^0$ of the action of $\theta$ on chain complexes. It is clear that the 2-functor $\RMon(\ba,U-)$ decomposes into subfunctors labelled by these homological kernels, and our main result states that this decomposition corresponds to an instance of multi-representability.

\begin{thmA}
Consider $\ba\in\RMon$. 
\begin{enumerate}
\item The 2-functor $\RMon(\ba,U-):\Tens\to\Cat$ is multi-representable, that is (for some set $\Gamma$ and tensor categories $\bT_\gamma,\gamma\in\Gamma$)
$$\RMon(\ba,U-)\;\simeq\;\coprod_{\gamma\in\Gamma}\Tens(\bT_\gamma,-).$$
\item The set $\Gamma:=\Spec_U\ba$ can be identified with the set of subfunctors of $K^b\ba(-,\unit)$ which appear as homological kernels; and there exist no inclusions between the subfunctors in $\Gamma$. A functor $\theta\in \RMon(\ba,U\bT)$ factors via $\bT_\gamma$ if its homological kernel is $\gamma$.
\item The category $\ba$ admits a weak abelian envelope if and only if precisely one subfunctor of $K^b\ba(-,\unit)$ appears as a homological kernel (meaning $\Spec_U(\ba)$ is a singleton).
\item $\Tens$ is a multi-reflective sub-2-category of $\RMon$.
\end{enumerate}
\end{thmA}
Parts (3) and (4) are tautological consequences of (1) and (2). The latter results are proved in Theorem~\ref{ThmUniTensork}.

\subsection*{Organisation of the paper}

In Section~\ref{SecPrel} we recall some preliminary notions. The remainder of the paper is divided into Parts~\ref{Part1} and~\ref{Part2}. Part~\ref{Part1} deals with the general theory of homological topologies and Part~\ref{Part2} applies that theory to our case of interest; tensor categories. 

Part \ref{Part1}: In Section~\ref{SecPrex}, we study prexact functors, which are functors for which the left Kan extension to the presheaf category is exact. In particular we use them to show that the set of Grothendieck topologies on a given category appears as a labelling set for a multi-representable 2-functor. In Section~\ref{SecHT} we introduce homological topologies to extend this result and in Section~\ref{SecMon} we discuss the generalisation of the previous result to monoidal functors.

Part \ref{Part2}: In Section~\ref{LocAbEnvTheory} we prove Theorem~A above. In Section~\ref{Examples} we explore examples. In particular we demonstrate that in certain cases $\Spec_U(\ba)$ is a singleton, giving new sources of examples of abelian envelopes and that in other cases $\Spec_U(\ba)$ can be an infinite set. We also realise the abelian categories of motives recently constructed in \cite{SchMot} as instances of our theory. Finally, in Appendix~\ref{App} we discuss some variations of the universal monoidal categories from \cite{Deligne} which are used in Section~\ref{Examples}.


\section{Preliminaries}\label{SecPrel}

\subsection{Abelian categories}

\subsubsection{}We follow the conventions from \cite{Gro} and use conditions AB1-AB5 from \cite[\S 1]{Gro}. An abelian category is an additive category satisfying AB1 and AB2. An AB3 category is an abelian category which admits all (small) coproducts, {\it i.e.} a cocomplete abelian category. An AB5 category is an AB3 category in which direct limits of short exact sequences are exact. A Grothendieck category is an AB5 category which admits a generator in the sense of \cite[\S 1.9]{Gro}.

\subsubsection{} A full additive subcategory $\bB$ of an abelian category $\bA$ is an abelian subcategory if $\bB$ contains all kernels and cokernels in $\bA$ of morphisms between objects in $\bB$. An AB3 subcategory of an AB3 category is an abelian subcategory closed under coproducts (and hence all small colimits).

\begin{definition}\label{DefMin}
An additive functor $\phi:\ba\to\bA$ from a small additive category $\ba$ to an AB3 category $\bA$ is {\bf AB3-tight} if for every AB3 subcategory $\bB\subset \bA$ which contains the image of $\phi$, we have $\bB=\bA$.
\end{definition}

The following lemmata are obvious consequences of the definitions, but it will be useful to have them spelled out.
\begin{lemma}\label{LemTight}
For a small additive category $\ba$ and AB3 categories $\bA$ and $\bB$, assume we have a commutative diagram of functors,
$$\xymatrix{
&\ba\ar[ld]\ar[d]^\phi\\
\bB\ar[r]^F&\bA,
}
$$
where $\phi$ is AB3-tight and $F$ is fully faithful, exact and cocontinuous. Then $F$ is an equivalence.
\end{lemma}

\begin{lemma}\label{LemAB5}
Consider a morphism $\oplus_{\beta\in B}M_\beta\to N$ in an AB5 category, and denote its kernel by $K$. For every finite subset $E\subset B$, denote by $K_E\subset K$ the kernel of $\oplus_{\beta\in E}M_\beta\to N$. Then we have $K=\cup_E K_E$.
\end{lemma}
\begin{proof}
Since the sequence
$$0\to\varinjlim_E K_E\to \bigoplus_{\beta\in B}M_\beta\to N$$
is the direct limit of the obvious left exact sequences, it is left exact. Therefore, $\varinjlim_E K_E\to K$ is an isomorphism.
\end{proof}

\subsection{Additive Grothendieck topologies}\label{SecGT} We refer to \cite{BQ} for the theory of Grothendieck topologies on enriched categories. We are interested only in additive Grothendieck topologies on preadditive categories. For this case the required notions from \cite{BQ} are written out in detail in \cite[\S 1.4]{PreTop}.

We fix the following notation. For a small preadditive category $\ba$, we denote the set of Grothendieck topologies on $\ba$ by $\Top(\ba)$. For $\cT\in\Top(\ba)$, we denote the (Grothendieck) category of $\cT$-sheaves on $\ba$ by $\Sh(\ba,\cT)$ and the composition of the Yoneda embedding $\Yon$ with the (exact) sheafification functor $\tS$ is denoted as
$$\tZ\,:\,\ba\xrightarrow{\Yon}\PSh\ba\xrightarrow{\tS}\Sh(\ba,\cT).$$

\subsection{Tensor categories}

Fix a field $k$. 

\subsubsection{}\label{Defk}We say that an essentially small $k$-linear additive monoidal category $(\ba,\otimes,\unit)$ with $k$-linear tensor product is a {\bf rigid monoidal category over $k$} if
\begin{enumerate}
\item The morphism $k\to\End(\unit)$ is an isomorphism.
\item Every object $X\in\ba$ has a left and right dual $X^\ast$ and ${}^\ast X$.
\end{enumerate}
\noindent Recall that a left dual of $X$ is actually a triple $(X^\ast,\ev_X,\co_X)$ with morphisms $\ev_X:X^\ast\otimes X\to\unit$ and $\co_X:\unit\to X\otimes X^\ast$ satisfying the snake relations.
If moreover,
\begin{enumerate}
\item[(3)] $\ba$ is abelian,
\end{enumerate}
we say that $(\ba,\otimes,\unit)$ is a {\bf tensor category over $k$}. It then follows that $\unit$ is simple, see for instance \cite[1.1.5]{CEOP}. A standard example is the tensor category $\Vecc_k$ of finite-dimensional vector spaces. A tensor category over $k$ is {\bf artinian} if all objects have finite length (it then follows that all morphism spaces are finite-dimensional too).

We will refer to exact $k$-linear monoidal functors between tensor categories (over $k$ or field extensions $K/k$) as {\bf tensor functors}. By \cite[Theorem~2.4.1]{CEOP} or \cite[Theorem~4.4.1]{PreTop}, we can replace `exact' with `faithful' in this definition.

We consider the 2-category $\Tens=\Tens_k$ of tensor categories over $k$, tensor functors and monoidal natural transformations. By $\Tens^{\uparrow}=\Tens_k^{\uparrow}$ we mean the 2-category of tensor categories over all field extensions $K/k$, tensor functors and monoidal natural transformations.

We also consider the 2-categories $\RMon=\RMon_k$ (resp. $\RMon^{\uparrow}=\RMon^{\uparrow}_k$) of rigid monoidal categories over $k$ (resp. over field extensions $K/k$), $k$-linear {\em faithful} monoidal functors and monoidal natural transformations. 

By reformulating the definitions from \cite{EHS, CEH, CEOP} we have the following terminology.
\begin{definition} Let $\ba$ be a rigid monoidal category over $k$.
\begin{enumerate}
\item The category $\ba$ admits a weak abelian envelope if the 2-functor $\RMon(\ba,-):\Tens\to\Cat$ is representable. The {\bf weak abelian envelope} is then given by the faithful tensor functor $\ba\to\bT$ for the representing tensor category $\bT$ which induces the equivalence $\RMon(\ba,-)\simeq\Tens(\bT,-)$.
\item A weak abelian envelope as in (1) is an {\bf abelian envelope} if $\ba\to\bT$ is full.
\end{enumerate}
\end{definition}

\subsection{Notation and conventions}

For an additive category $\bA$, we denote by $\bA^{\mZ}$ the category $\bigoplus_{i\in\mZ}\bA$ and write $X\langle i\rangle$ for the object $X\in\bA$ interpreted in to be in the $i$-th copy of $\bA$ in $\bA^{\mZ}$.

For a ring $R$, we denote by $R\free$ the category of free (left) $R$-modules of finite rank. By $R\Mod$ we denote the category of all $R$-modules. We denote by $\Ab$ the category of abelian groups and by $\Ab_f$ the category of finitely generated abelian groups.

In Part 1, for two pradditive categories $\ba,\bb$, we denote by $[\ba,\bb]$ the category of additive functors $\ba\to\bb$. For instance $\PSh\ba=[\ba^{\op},\Ab]$. We will use subscripts to specify subcategories. For instance, if $\ba$ is finitely complete, $[\ba,\bb]_{lex}$ is the category of left exact functors and if $\ba$ is triangulated and $\bb$ abelian, $[\ba,\bb]_{hom}$ is the category of homological functors. In Part 2 the same notation will be used for $k$-linear functors.

For a category $C$ of functors between two (symmetric) monoidal categories, $C^\otimes$ ($C^{s\otimes}$) is the category of (symmetric) monoidal functors, and monoidal natural transformations, for which the underlying functors belong to~$C$.

We largely ignore set theoretic issues. In Part 1, it would be more precise to work with additive essentially small $\mathcal{U}$-categories and abelian $\mathcal{U}$-categories, and then extend the Grothendieck universe $\mathcal{U}$ when considering 2-categories. For Part 2 this becomes irrelevant as we only need essentially small 1-categories.


\vspace{2mm}

\part{Homological Grothendieck topologies}\label{Part1}

\section{Prexact functors and Grothendieck topologies}\label{SecPrex}
Let $\ba$ and $\bb$ be essentially small preadditive categories and let $\bC$ denote an AB3 (cocomplete abelian) category. 
\subsection{Definitions}

\subsubsection{}\label{DefLan}For any additive functor $\theta:\ba\to\bC$, we can consider the left Kan extension $\Lan_{\Yon}\theta:\PSh\ba\to\bC$ along the Yoneda embedding $\Yon:\ba\to\PSh\ba$, which we will typically denote by~$\Theta$. For our purposes, we can actually define $\Theta$ as the unique, up to isomorphism, cocontinuous functor $\Theta:\PSh\ba\to\bC$ for which $\Theta\circ\Yon$ is isomorphic to $\theta$, see for instance \cite[Proposition~2.5.2]{PreTop}, or as the left adjoint to $\bC\to\PSh\ba,\,M\mapsto \bC(\theta-,M)$. 

For an additive functor $u:\ba\to\bb$, we consider 
$$u_!=-\otimes_{\ba}\bb:=\Lan_{\Yon}(\Yon\circ u)\,:\,\PSh\ba\to\PSh\bb.$$
By the above, $u_!$ is the left adjoint of $-\circ u:\PSh\bb\to\PSh\ba$.

\begin{definition}\label{DefFlat}
\begin{enumerate}
\item An additive functor $\theta:\ba\to\bC$ is {\bf prexact} if the cocontinuous functor $\Theta=\Lan_{\Yon}\theta:\PSh\ba\to\bC$ is (left) exact.
\item An additive functor $u:\ba\to\bb$ is {\bf flat} if the cocontinuous functor $u_!:\PSh\ba\to\PSh\bb$ is (left) exact, or equivalently if $\Yon\circ u$ is prexact. 
\end{enumerate}
\end{definition}


By \ref{DefLan}, we can alternatively define prexact functors as the restrictions of exact cocontinuous functors $\PSh\ba\to\bC$ along $\Yon$. This yields the following examples.
\begin{example}\label{ExPrex}
\begin{enumerate}
\item The functor $\tZ:\ba\to\Sh(\ba,\cT)$ is prexact for every $\cT\in\Top(\ba)$. In fact, these will be the `universal prexact functors'.
\item If $\theta:\ba\to\bC$ is prexact, then so is $F\circ \theta$ for every exact cocontinuous functor $F$ between AB3 categories. If $F$ is also faithful then conversely $F\circ \theta$ prexact implies that $\theta$ must be prexact.
\item If $\theta:\ba\to\bC$ is prexact and $u:\bb\to\ba$ is flat, then $\theta\circ u$ is prexact.
\end{enumerate}
\end{example}

\subsubsection{The topology of a prexact functor}\label{Ttheta} Consider a prexact functor $\theta:\ba\to\bC$. We denote by $\cT_{\theta}$ the following covering system on $\ba$. For $X\in\ba$ and $R\subset\ba(-,X)$, we have $R\in\cT_\theta(X)$ if and only if one of the following equivalent conditions is satisfied
\begin{enumerate}
\item $\Theta(R)\;\to\;\theta(X)$ is an isomorphism;
\item $R$ contains a collection of morphisms $Y_\beta\to X$ such that $\oplus_\beta\theta(Y_\beta)\to\theta(X)$ is an epimorphism.
\end{enumerate}
Then $\cT_\theta$ is a Grothendieck topology on $\ba$, as follows either directly from the definition (\cite[1.4.2]{PreTop}) or because, using the notation of \cite[\S 2.2]{PreTop}, we have $\cT_{\theta}=\mathrm{top}(\cS(\Theta))$.

\begin{example}\label{ExTopPrex}
\begin{enumerate}
\item For a given Grothendieck topology $\cT$ on $\ba$ and $\tZ:\ba\to\Sh(\ba,\cT)$, we have $\cT_{\tZ}=\cT$, by \cite{BQ}, see \cite[Theorem~1.4.5]{PreTop}.
\item For a prexact functor $\theta:\ba\to\bC$ and a faithful exact cocontinuous functor $F:\bC\to\bB$ to a second AB3 category $\bB$, it follows that $\cT_{\theta}=\cT_{F\circ \theta}$.
\end{enumerate}
\end{example}


For sets $\{X_\alpha\,|\, \alpha\in A\}$ and $\{Y_\beta\,|\, \beta\in B\}$ of objects in $\ba$ we call elements of
$$\PSh\ba\left(\bigoplus_\beta \Yon(Y_\beta),\bigoplus_\alpha\Yon(X_\alpha)\right)\;\simeq\; \prod_\beta\bigoplus_\alpha\ba(Y_\beta,X_\alpha)$$
`formal morphisms' and write them as $\sqcup_\beta Y_\beta\to\sqcup_\alpha X_\alpha$. Their composition is defined canonically.

\begin{lemma}\label{ThmLem}
For an additive functor $\theta:\ba\to\bC$ the following are equivalent:
\begin{enumerate}[label=(\alph*)]
\item $\theta$ is prexact.
\item For each formal morphism $f:\sqcup_\beta Y_\beta\to \sqcup_\alpha X_\alpha$ in $\ba$, the sequence
$$\bigoplus_{g:Z\to \sqcup_\beta Y_\beta, f\circ g=0}\theta(Z)\;\to\;\bigoplus_\beta \theta(Y_\beta)\;\to \;\bigoplus_{\alpha}\theta(X_\alpha)$$   
is exact (acyclic) in $\bC$.
\end{enumerate}
\end{lemma}
\begin{proof}
Assume first that $\theta$ is prexact. The sequence 
$$\bigoplus_{g:Z\to \sqcup_\beta Y_\beta, f\circ g=0}\Yon(Z)\;\to\; \bigoplus_\beta\Yon(Y_\beta)\;\to \;\bigoplus_\alpha\Yon(X_\alpha)$$
is tautologically exact in $\PSh\ba$. The fact that $\Theta$ is exact and cocontinuous then implies that the sequence in (b) is exact. Hence (a) implies (b).

If (b) is satisfied, it follows that for every $N\in \PSh\ba$ we can construct an exact sequence
$$\bigoplus_\gamma \Yon(Z_\gamma)\;\to\; \bigoplus_\beta\Yon(Y_\beta)\;\to\; \bigoplus_\alpha \Yon(X_\alpha)\;\to\; N\; \to\; 0$$
in $\PSh\ba$ which is sent to an exact sequence in $\bC$ by $\Theta$. Indeed, we can start from a projective presentation of $N$ and add an additional term by applying (b).

That $\Theta$ is exact (so that (b) implies (a)) is then a standard consequence of this observation. Indeed, for a short exact sequence 
$$0\to N\to N'\to N''\to 0$$
in $\PSh\ba$ we can take such a partial projective resolution of $N$ and $N''$, which also induces one for $N'$. This allows us to enlarge the above exact row to a commutative diagram where every row except the one displayed above is {\em split} exact and which has exact columns. Applying $\Theta$ then preserves exactness of all rows, except that, a priori, 
$$0\to \Theta N\to \Theta N'\to \Theta N''\to 0$$ is only right exact. By assumption, the right column is exact and, by right exactness of $\Theta$, also the middle column is exact when we ignore the lowest term.
It follows from diagram chasing that also the displayed row is exact.
\end{proof}


\subsection{Multi-representability for prexact functors and a partial function}
The two results in this section will be completed in Section~\ref{SecHT} when we introduce homological Grothendieck topologies.


\subsubsection{}\label{Def2CatGro} Consider the 2-category $\ABt$ which has as objects AB3 categories, as 1-morphisms faithful exact cocontinuous functors and as 2-morphisms all natural transformations.

For an AB3 category $\bC$, we denote by $[\ba,\bC]_{prex}$ the category of prexact functors $\ba\to\bC$.
By Example~\ref{ExPrex}(2), we can interpret this as a 2-functor
$$[\ba,-]_{prex}\,:\;\ABt\to\Cat.$$
Unless $\ba$ is the zero category, this 2-functor is not representable, as Example~\ref{ExTopPrex}(2) implies the 2-functor decomposes as
\begin{equation}\label{deca[}
[ba,-]_{prex}\;=\;\coprod_{\cT\in \Top\ba} [\ba,-]_{prex,\cT}
\end{equation}
where $[\ba,-]_{prex,\cT}$ is the category of prexact functors $\phi$ with $\cT_\phi=\cT$. 

\begin{remark}
The decomposition \eqref{deca[} ignores the additive structure of $[\ba,\bC]_{prex}$. In fact, for $F_i\in [\ba,\bC]_{prex,\cT_i}$ with $i\in\{1,2\}$, we have $$F_1\oplus F_2\;\in\; [\ba,\bC]_{prex, \cT_1\cap\cT_2}.$$
\end{remark}

\begin{theorem}\label{PropFlatRep}
For each Grothendieck topology $\cT$ on $\ba$, the 2-functor $$ [\ba,-]_{prex,\cT}:\;\ABt\to\Cat$$ is represented by $\Sh(\ba,\cT)$. More concretely, composition with $\tZ:\ba\to\Sh(\ba,\cT)$ yields an equivalence
$$\ABt(\Sh(\ba,\cT),\bC)\;\stackrel{\sim}{\to}\; [\ba,\bC]_{prex,\cT}.$$
\end{theorem}
\begin{proof}

We let $\cS$ be the class of all formal sequences 
$$\sqcup_{\gamma} Z_\gamma\,\to\, \sqcup_\beta Y_\beta\,\to\, X$$
in $\ba$ for which $\oplus\Yon(Z_\gamma)\to\oplus \Yon(Y_\beta)\to \Yon(X)$ is acyclic in $\PSh\ba$ and for which  $\oplus\tZ(Y_\beta)\to \tZ(X)$ is an epimorphism (or equivalently the sieve generated by $\sqcup_\beta Y_\beta\to X $ is in the topology $\cT$). This is the pretopology $\cS=\pre'(\cT)$ from \cite[\S 2.2.4]{PreTop}.

By \cite[Proposition~2.5.2]{PreTop}, composition with $\tZ$ yields an equivalence
\begin{equation}\label{StartEq}-\circ\tZ\,:\;[\Sh(\ba,\cT),\bC]_{cc}\;\stackrel{\sim}{\to}\; [\ba,\bC]_{\cS},\end{equation}
where the left-hand side is the category of cocontinuous functors and the right-hand side is the category of additive functors $\theta:\ba\to\bC$ for which $\oplus\theta(Z_\gamma)\to\oplus \theta(Y_\beta)\to \theta(X)\to0$ is exact for each sequence in $\cS$. We will prove that this equivalence \eqref{StartEq} restricts to the one in the theorem.

By definition, $\ABt(\Sh(\ba,\cT),\bC)$ is a full subcategory of $[\Sh(\ba,\cT),\bC]_{cc}$. We claim that $[\ba,\bC]_{prex,\cT}$ is a (full) subcategory of $[\ba,\bC]_{\cS}$. Indeed, consider $\theta\in[\ba,\bC]_{prex}$. Then $\Theta$ is exact and cocontinuous, so it follows that $\oplus\theta(Z_\gamma)\to\oplus \theta(Y_\beta)\to \theta(X)$ is acyclic. If furthermore $\theta\in[\ba,\bC]_{prex,\cT}$ it follows that $\oplus\theta(Y_\beta)\to \theta(X)$ is an epimorphism. Now we prove that equivalence  restricts to the relevant subcategories.

By applying equivalence \eqref{StartEq} to $\theta\in[\ba,\bC]_{prex,\cT}$ we find an essentially unique cocontinuous functor $T: \Sh(\ba,\cT)\to\bC$ with $T\circ\tZ\simeq \theta$. We have to show that $T$ is exact and faithful.

Since we have $\tZ=\tS\circ\Yon$ and $\theta\simeq \Theta\circ\Yon$, we find $T\circ\tS\simeq \Theta$. Denote by $\tI$ the inclusion of the full subcategory $\Sh(\ba,\cT)$ in $\PSh\ba$, so $\tS\dashv \tI$. Using $\tS\circ \tI\simeq\id$ then implies that $T\simeq \Theta\circ\tI$. The right-hand side of the latter isomorphism is left exact, which shows that $T$ is (left) exact.

Now assume $T(F)=0$ for $F\in\Sh(\ba,\cT)$, which is equivalent to $\Theta(\tI(F))=0$. If $F$ is non-zero there exists a non-zero morphism $\Yon(X)\to\tI(F)$ for some $X\in\ba$, and hence an acyclic sequence
$$\bigoplus_\beta\Yon(Y_\beta)\to\Yon(X)\to \tI(F)$$
in $\PSh\ba$. Applying the exact continuous functor $\Theta$ shows that $\oplus_\beta\theta(Y_\beta))\to\theta(X)$ is an epimorphism, which means that $\oplus_\beta\tZ(Y_\beta)\to\tZ(X)$ is also an epimorphism since $\cT_{\tZ}=\cT=\cT_\theta$. It follows in particular that the morphism $\tZ(X)\to F$ obtained by adjunction from $\Yon(X)\to \tI(F)$ is zero, a contradiction. We have proved that $T$ is exact and faithful as desired.

Now assume that we have an exact faithful and cocontinuous functor $F:\Sh(\ba,\cT)\to\bC$. Then $F\circ\tZ$ is prexact since $\tZ$ is prexact, see Example~\ref{ExPrex}, and faithfulness of $F$ shows that $\cT_{F\circ \tZ}=\cT_{\tZ}=\cT$, see Example~\ref{ExTopPrex}.
\end{proof}

The following consequence is elementary, but it will be useful to have it written out.
\begin{corollary}\label{CorTriv}
Consider prexact functors $\xi_i:\ba\to \bC_i$ to AB3 categories $\bC_i$ for $i\in\{1,2\}$. If $\cT_{\xi_1}=\cT_{\xi_2}$, then we also have $\ker\xi_1=\ker\xi_2$.
\end{corollary}

\subsubsection{A partial function}\label{DefPartFun}
Consider an additive functor $u:\ba\to\bb$. We define a (inclusion preserving) partial function
$$u_{t}:\Top(\bb) \rightharpoonup \Top(\ba).$$
For a Grothendieck topology $\cT$ on $\bb$, consider $\tZ:\bb\to\Sh(\bb,\cT)$. In case $\tZ\circ u$ is prexact, we define $u_{t}(\cT)$ as $\cT_{\tZ\circ u}$. In other words, when $u_t(\cT)$ is defined, then $R\subset \ba(-,A)$ is in $u_t(\cT)$ if and only if it contains morphisms $f_\beta :B_\beta\to A$ for which the sieve on $u(A)$ generated by $u(f_\beta)$ is in $\cT$. By Theorem~\ref{PropFlatRep}, in case $u_t(\cT)$ is defined, there exists a commutative diagram of functors
\begin{equation}\label{CompareSh}\xymatrix{
\ba\ar[rr]^u\ar[d]^{\tZ}&&\bb\ar[d]^{\tZ}\\
\Sh(\ba,u_t(\cT))\ar@{-->}[rr]&&\Sh(\bb,\cT)
}
\end{equation}
where the lower arrow is a faithful exact cocontinuous functor.

By Example~\ref{ExPrex}(3), $u_t$ is a function (defined everywhere) if and only if $u$ is flat.





\subsection{Criteria for flat and prexact functors}\label{Crit}
In this section we assume that $\bC$ is an AB5 category and that $\ba$ and $\bb$ are additive. Then we can improve on Lemma~\ref{ThmLem}.
\begin{theorem}\label{ThmFlat}
For an additive functor $\theta:\ba\to\bC$ the following are equivalent:
\begin{enumerate}[label=(\alph*)]
\item $\theta$ is prexact.
\item For each $f:Y\to X$ in $\ba$, the sequence
$$\bigoplus_{g:Z\to Y, f\circ g=0}\theta(Z)\;\to\; \theta(Y)\;\to \;\theta(X)$$   
is exact (acyclic) in $\bC$.
\end{enumerate}
\end{theorem}
\begin{proof}
Based on Lemma~\ref{ThmLem}, it is sufficient to prove that for an AB5 category $\bC$, exactness of all sequences in (b) implies exactness of all sequences in \ref{ThmLem}(b).

We will prove this in increasing generality. Denote by $A$ and $B$ the sets to which $\alpha$ and $\beta$ in \ref{ThmLem}(b) belong.
 Since $\ba$ and $\theta$ are additive, when $A$ and $B$ are both finite, exactness of the sequence in \ref{ThmLem}(b) follows from (b). Now assume that $B$ is finite but allow $A$ to be infinite. By definition of formal morphisms, we can still interpret $f$ as an actual morphism in $\ba$ and \ref{ThmLem}(b) remains exact. Finally we also allow $B$ to be infinite. Denote by $K$ the kernel of the right morphism of the sequence in \ref{ThmLem}(b) and by $I\subset K$ the image of the left morphism. For every finite subset $E\subset B$, by the previous case we know that the kernel $K_E\subset K$ of $\bigoplus_{\beta\in E} \theta(Y_\beta)\to\bigoplus_\alpha \theta(X_\alpha)$ is contained in $I$. By Lemma~\ref{LemAB5}, $K=\cup_E K_E$, and hence $I=K$. It follows indeed that \ref{ThmLem}(b) is exact.\end{proof}
 
 \begin{remark}
 The condition that $\bC$ be AB5 abelian is necessary for Theorem~\ref{ThmFlat}. Indeed, it suffices to consider the inclusions of small abelian subcategories of AB3 (or even AB4) categories. These always satisfy \ref{ThmFlat}(b), but need not be prexact. For a concrete example, take $\Ab_f^{\op}\hookrightarrow\Ab^{\op}$. Indeed, a direct limit of short exact sequences in $\Ab_f^{\op}$ is exact when considered in $\PSh(\Ab_f^{\op})$, however, it is not necessarily (left) exact in $\Ab^{\op}$, preventing exactness of $\PSh(\Ab_f^{\op})\to \Ab^{\op}$.
\end{remark}

\begin{remark}\label{RemEnd}
By taking a direct sum of $f$ with the zero morphism $X\to Y$, we can restrict to the case $X=Y$ in \ref{ThmFlat}(b).
\end{remark}

For $\bc$ an essentially small abelian category, we denote its ind-completion by $\Ind\bc$. With slight abuse of notation we will call an additive functor $\ba\to\bc$ prexact if the composite $\ba\to\bc\to\Ind\bc$ is prexact.

\begin{corollary}\label{CorFlat}
Let $\bc$ be an essentially small abelian category. An additive functor $\theta:\ba\to\bc$ is prexact if and only if for each $f:Y\to X$ in $\ba$,  there exists $g:Z\to Y$ with $f\circ g=0$ such that
the sequence
$$\theta(Z)\xrightarrow{\theta(g)}\theta(Y)\xrightarrow{\theta(f)} \theta(X)$$
is exact (acyclic) in $\bc$.
\end{corollary}
\begin{proof}
If there exists $Z$ as in the corollary, then clearly the sequence in \ref{ThmFlat}(b) is exact, so prexactness of $\theta$ follows from Theorem~\ref{ThmFlat}. On the other hand if $\theta$ is prexact then the sequence in \ref{ThmFlat}(b) is exact. Moreover, since $\bc$ is an abelian subcategory of $\Ind\bc$, the kernel $K$ of $\theta(f)$ is compact in $\Ind\bc$. Hence we can restrict to a finite coproduct in $\bigoplus_{g:Z\to Y, f\circ g=0}\theta(Z)\tto K$ (note that coproducts in \ref{ThmFlat}(b) are taken in $\Ind\bc$) while retaining an epimorphism. By additivity of $\ba$ we arrive at a single $Z\to Y$.
\end{proof}

\begin{prop}\label{PropWK}
 If $\ba$ has weak kernels, then the following are equivalent conditions on an additive functor $\theta:\ba\to\bC$:
\begin{enumerate}[label=(\alph*)]
\item $\theta$ is prexact. 
\item For every morphism $f:Y\to X$, there is a weak kernel $K\to Y$ for which
$$\theta(K)\to\theta(Y)\xrightarrow{\theta(f)} \theta(X) $$
is acyclic.
\item For every morphism $f:Y\to X$, the sequence
$$\theta(K)\to\theta(Y)\xrightarrow{\theta(f)} \theta(X) $$
is acyclic for every weak kernel $K\to Y$ of $f$.
\end{enumerate}

\end{prop}
\begin{proof}
Clearly (c) implies (b). That (b) implies (a) follows from Theorem~\ref{ThmFlat}. That (a) implies (c) follows again from Theorem~\ref{ThmFlat} and the universality of a weak kernel.
\end{proof}

\begin{corollary}\label{TrFlat}
\begin{enumerate}
\item Assume that $\ba$ is finitely complete ($\ba$ has kernels) and $\theta:\ba\to\bC$ is an additive functor. Then $\theta$ is prexact if and only if it is left exact (i.e. $\theta$ preserves finite limits).
\item Let $\bt$ be an essentially small triangulated category and $\theta:\bt\to\bC$ an additive functor. Then $\theta$ is prexact if and only if it is homological.
\end{enumerate}
\end{corollary}
\begin{proof}
By definition, prexact functors are always left exact. Part (1) therefore follows easily from Proposition~\ref{PropWK}.


Now we prove part (2). By the rotation axiom of triangles, $\theta$ is homological if and only if 
$$\theta(Z)\xrightarrow{\theta(g)}\theta(Y)\xrightarrow{\theta(f)}\theta(X)$$
is acyclic for every distinguished triangle
$$ Z\xrightarrow{g} Y\xrightarrow{f} X\,\to\, Z[1].$$

Since $g$ is a weak kernel of $f$, and since every morphism can be completed to a distinguished triangle, \ref{PropWK} (b)$\Rightarrow$(a) shows that homological functors are prexact. Similarly, \ref{PropWK} (a)$\Rightarrow$(c) shows that prexact functors are homological.
\end{proof}



\begin{remark}
 Corollary~\ref{TrFlat}(2) can alternatively be derived from Freyd's universal homological functor, see~\cite{Freyd2, Verdier}.

\end{remark}

We conclude this section by applying the above results to obtain criteria for flatness.

\begin{corollary}\label{CorCorCor}
The following conditions are equivalent on an additive functor $u:\ba\to\bb$.
\begin{enumerate}[label=(\alph*)]
\item $u$ is flat.
\item For every pair of a morphism $f:Y\to X$ in $\ba$ and a morphism $b:B\to u(Y)$ in $\bb$ with $u(f)\circ b=0$, there exist morphisms $g:Z\to Y$ with $f\circ g=0$ and $h:B\to u(Z)$ such that $b=u(g)\circ h$.
\item For every object $B\in\bb$, the comma category $B/u$, with objects given by the morphisms $\{B\to u(A)\,|\, A\in\ba\}$, is cofiltered.
\end{enumerate}
\end{corollary}

\begin{remark}By Theorem~\ref{ThmFlat}, a functor $\theta:\ba\to \Ab$ is prexact if and only if the category of elements of $\theta$ is cofiltered. The latter category is equal to the comma category $\mZ/\theta$. 
By Corollary~\ref{CorCorCor}, this demonstrates the (easily directly verifiable) fact that a flat functor to $\Ab_f$ is prexact. The converse is not true: the inclusion $\ba\hookrightarrow \Ab_f$ of the full additive subcategory generated by the cyclic group of order 4 is prexact, but not flat.
\end{remark}

\begin{corollary}\label{CorWKFlat}
Assume that $\ba$ has weak kernels. The following conditions are equivalent on an additive functor $u:\ba\to\bb$.
\begin{enumerate}[label=(\alph*)]
\item $u$ is flat. 
\item For every morphism $f:Y\to X$, there is a weak kernel $K\to Y$ for which
$u(K)$ is a weak kernel of $u(f)$.
\item $u$ sends all weak kernels to weak kernels.
\end{enumerate}
\end{corollary}


\section{The kernel category and homological topologies}\label{SecHT}
Let $\ba$ be an essentially small additive category.
\subsection{Some categories with weak kernels}

\subsubsection{}\label{secdiag}We will work extensively with the following categories and functors:
$$\xymatrix{ 
\ba\ar[r]&K_+^b\ba\ar[r]\ar[rd]& K^b\ba\\
&&\Noy\ba.
}$$
Here $K^b\ba$ is the bounded homotopy category of $\ba$ (the quotient of the category of bounded complexes by the chain homotopy relation) and $K_+^b\ba$ is its full subcategory of complexes $(\cX,d)$ for which $\cX^i=0$ whenever $i<0$. The functor $\ba\to K_+^b\ba$ is the obvious embedding of $\ba$ as the full subcategory of complexes contained in degree $0$. We use the same notation for an object $X\in\ba$ interpreted in $\ba$ or as a chain complex contained in degree $0$. 
We take the convention that the shift functor $[1]$ on $K^b\ba$ acts as
$$(\cX[1])^i\;=\;\cX^{i+1},$$ 
in particular $X[i]$ for $X\in\ba$ is a chain complex contained in degree $-i$.


\begin{definition}
The category $\Noy\ba$ has as objects morphisms in $\ba$. For morphisms $f:X^0\to X^1$ and $g:Y^0\to Y^1$ in $\ba$, the set $\Noy\ba(f,g)$ consists of the equivalence classes of morphisms $\alpha: X^0\to Y^0$ for which $g\circ \alpha$ factors via $f$ 
\begin{equation}\label{ExNoy}\xymatrix{
X^0\ar[rr]^f\ar[d]^{\alpha}&&X^1\ar@{-->}[d]\\
Y^0\ar[rr]^g&&Y^1,\\
}\end{equation}
where $\alpha$ and $\alpha'$ are equivalent if $\alpha-\alpha'$ factors via $f$. Denote by
$$\tN :\ba\to\Noy\ba$$
the fully faithful functor which sends $A\in\ba$ to $A\to 0$.
\end{definition}

Composition in $\Noy\ba$ is defined in the obvious way and $\Noy\ba$ is additive. We can also define the group $\Noy\ba(f,g)$ via the following diagram:
$$\xymatrix{
\ba(X^0,Y^0)\ar[rr]^{\phi_1=(g\circ-)}&&\ba(X^0,Y^1)&& \\
\ba(X^1,Y^0)\ar[u]^{\phi_2=(-\circ f)}&&\ba(X^1,Y^1)\ar[u]^{\phi_3=(-\circ f)}&& \Noy\ba(f,g)=\phi_1^{-1}(\im\phi_3)/\im\phi_2.
}$$
For example, for morphisms $f,g$ and an object $A$ in $\ba$, we have
\begin{equation}\label{EqNAf}
\Noy\ba(f,\tN A)=\coker\,\ba(f,A)\qquad\mbox{and}\qquad \Noy\ba(\tN A,g)=\ker\ba(A,g).
\end{equation}

The functor $K_+^b\ba\to\Noy\ba$ in the diagram in \ref{secdiag} sends $(\cX,d)$ to $d^0: \cX^0\to \cX^1$ and a class of chain maps with representative $a:\cX\to \cY$ to the class of morphisms in $\Noy\ba$ corresponding to $a^0:\cX^0\to\cY^0$.

The following proposition summarises how $\Noy\ba$ corresponds to several categories in the literature, for instance in~\cite{Be, Freyd1}. The equivalence between the categories in (2) and (3) is given in \cite[Corollary~3.9(i)]{Be}.
\begin{prop}
The category $\Noy\ba$ is equivalent to each of the following three categories.
\begin{enumerate}
\item The quotient of the subcategory of $K^b\ba$ of complexes contained in degree $0,1$ with respect to the full subcategory of complexes in degree $1$.
\item The quotient of the arrow category $\Arr\ba$ with respect to its full subcategory of retracts $i: A\to B$ in $\ba$.
\item The category $(\fp \ba)^{\op}$, where $\fp\ba$ stands for the category of finitely presented functors in the module category $[\ba,\Ab]$.
\end{enumerate}
\end{prop}

\begin{example}\label{ExNoyA}
If $\ba$ is the additive closure of the one object category of a finite-dimensional algebra $A$ over a field $k$, then $\Noy\ba$ is the category of finite-dimensional right $A$-modules.
\end{example}


 As observed in \cite{Be,Freyd1} the category $\Noy\ba$ has kernels, as it is obtained from $\ba$ by freely adjoining kernels.
 \begin{lemma}\label{LemK+K}
 \begin{enumerate}
\item The categories $K_+^b\ba$ and $K^b\ba$ have weak kernels.
\item The inclusion $K_+^b\ba\to K^b\ba$ is flat.
\item  The category $\Noy\ba$ has kernels and $K^b_+\ba\to \Noy\ba$ is flat.
\end{enumerate}
 \end{lemma}
 \begin{proof}
 Since $K^b\ba$ is triangulated it has weak kernels. Concretely, a weak kernel of $f:\cX\to \cY$ is given by $\Cone(f)[-1]\to\cX$. Moreover, if $\cX,\cY\in K_+^b\ba$, then $\Cone(f)[-1]\in K_+^b\ba$. Consequently, also $K_+^b\ba$ has weak kernels. This proves part (1) and, by Corollary~\ref{CorWKFlat}, also part (2).
 
It follows easily that for a morphism in $\Noy\ba$ represented by $\alpha$ as in \eqref{ExNoy},
the kernel is given by $(f,\alpha): X^0\to X^1\oplus Y^0$. That $K^b_+\ba\to \Noy\ba$ is flat follows from Corollary~\ref{CorWKFlat}.
\end{proof}

\begin{remark}
It is clear that $\tN :\ba\to \Noy\ba$ does not preserve the kernels that might already exist in $\ba$. 
In particular, when $\ba$ admits kernels it does {\bf not} follow that $\ba\simeq\Noy\ba$.
However, as proved in \cite[Lemma~3.3]{Be}, $\tN $ does preserve all cokernels that exist in $\ba$.
\end{remark}

\subsubsection{}Let $\bC$ be an abelian category. For an additive functor $\theta:\ba\to\bC$, we use the same notation for the associated functor $K^b\ba\to K^b\bC$ and we will introduce the following functors
$$\xymatrix{ 
\ba\ar[r]&K_+^b\ba\ar[r]\ar[rd]& K^b\ba\ar[r]^{\theta_{\Delta}^{\mZ}} \ar[rd]^{\theta^0_{\Delta}}&\bC^{\mZ}\\
&&\Noy\ba\ar[r]_{\vec{\theta}}& \bC.
}$$

We let $\vec{\theta}:\Noy(\ba)\to \bC$ be the functor
$$f\,\mapsto\, \ker \theta(f),$$
for an arbitrary morphism $f$ in $\ba$ interpreted as an object in $\Noy\ba$.
Similarly, $\theta_\Delta^0: K^b\ba\to\bC$  is the functor
$$(\cX,d)\,\mapsto\, H^0(\theta \cX)= \ker\theta(d^0)/\im\theta(d^{-1}).$$
The functor $\theta_{\Delta}^{\mZ}:K^b\ba\to\bC^{\mZ}$ is defined as  
$$(\cX,d)\,\mapsto\, \bigoplus_{i\in\mZ}H^i(\theta \cX)\langle i\rangle.$$
Finally, we will write $\theta^0_+$ for the functor $K_+^b\ba\to\bC$ which can be obtained by composition in two equivalent ways in the above diagram, meaning the functor
$$(\cX,d)\,\mapsto\, \ker\theta(d^0).$$

These functors have a number of universal properties. Part (1) below is essentially~\cite[Corollary~3.2]{Be}.
\begin{theorem}\label{Thm3Way}
Fix an abelian category $\bC$.
\begin{enumerate}
\item Composition with $\tN :\ba\to \Noy\ba$ yields an equivalence (with inverse $\theta\mapsto \vec{\theta}$)
$$[\Noy\ba,\bC]_{lex}\;\stackrel{\sim}{\to}\;[\ba,\bC].$$
\item Composition with $\ba\to K^b_+\ba$ yields an equivalence (with inverse $\theta\mapsto \theta^0_+$)
$$[K^b_+\ba,\bC]_{wker,0}\;\stackrel{\sim}{\to}\;[\ba,\bC]$$
where $[K^b_+\ba,\bC]_{wker,0}$ stands for the category of functors $\phi:K^b_+\ba\to\bC$ which send weak kernels to acyclic sequences and satisfy one of the following equivalent additional conditions:
\begin{enumerate}
\item $\phi(X[i])=0$ for all $X\in\ba$ and $i<0$;
\item $\phi(\cX[-1])=0$ for all $\cX\in K^b_+\ba$.
\end{enumerate}
 \item 
Composition with $\ba\to K^b\ba$ yields an equivalence (with inverse $\theta\mapsto \theta_\Delta^0$)
$$[K^b\ba,\bC]_{hom,0}\;\stackrel{\sim}{\to}\; [\ba,\bC],$$
where $[K^b\ba,\bC]_{hom,0}$ stands for the category of homological functors $\phi:K^b\ba\to\bC$  which satisfy one of the following equivalent additional conditions:
\begin{enumerate}
\item $\phi(X[i])=0$ for all $X\in\ba$ and $i\not=0$;
\item $\phi(\cX)=0$ for all $\cX\in K^b\ba$ with $\cX^0=0$.
\end{enumerate}

\end{enumerate}
If $\bC$ is an AB5 category, the respective homological conditions in the left-hand sides of each equivalence can be replaced by the condition of being prexact.
 
\end{theorem}
\begin{proof}
The interpretation in terms of prexact functors in case $\bC$ is AB5 follows from Proposition~\ref{PropWK} and Corollary~\ref{TrFlat}.

We start by proving part (1). It is readily verified that $\vec{\theta}$ is left exact, for $\theta:\ba\to\bC$.
It is then also clear that $\theta\mapsto \vec{\theta}$ yields a left inverse to $\phi\mapsto \phi\circ \tN $. It thus suffices to show that for every $\phi\in [\Noy\ba,\bC]_{lex}$, we have $\phi\simeq \overrightarrow{\phi\circ \tN }$ functorially.
A suitable isomorphism $\phi\Rightarrow \overrightarrow{\phi\circ \tN }$ follows from the sequence
$$0\to\phi(X^0\xrightarrow{f}X^1)\to \phi(\tN X^0)\to \phi(\tN X^1),$$
which is exact by left exactness of $\phi$.

Equivalence of (a) and (b) in part (2) can be proved by induction on the length of the complex. That $\theta^0_+$ maps one particular weak kernel (for a given morphism) to an acyclic sequence is obvious, from which the property follows for all weak kernels. That the restriction of $\theta^0_+$ to $\ba$ gives back $\theta$ is obvious, so we just need to prove that there is a natural isomorphism $\phi(\cX,d)\xrightarrow{\sim}\ker \phi(d^0)$. Firstly, we consider the canonical morphism $x:\cX\to (\cX^0\to\cX^1)$ for $\cX\in K_+^b\ba$. It is a weak kernel of a morphism to an object which $\phi$ sends to zero, hence  $\phi(\cX)\to \phi(\cX^0\to \cX^1)$ is an epimorphism. Moreover, there is a weak kernel of $x$ which is also sent to zero by $\phi$, which implies that $\phi(\cX)\to \phi(\cX^0\to \cX^1)$ is a monomorphism and hence an isomorphism. The same type of reasoning then shows that
$$\phi(\cX^1[-1])=0\to \phi(\cX^0\to \cX^1)\to \phi(\cX^0)\to \phi(\cX^1)$$
is exact, completing the proof.

 can be proved by a simplified version of the argument for part (3) below.

Finally we prove part (3). That for homological functors condition (a) implies (b) follows by induction on the length of a complex. Composition in one direction of the proposed inverses is obviously isomorphic to the identity. We claim the composition in the other direction is also the identity. Take therefore $\phi\in [K^b\ba,\bC]_{hom,0}$ and consider the following commutative diagram, for any $(\cX,d)\in K^b\ba$:
$$\xymatrix{
&&0&0\ar[d]\\
0\ar[r]&\phi(\cX)\ar[r]&\phi(\tau^{\le 0}\cX)\ar[r]\ar[u]&\phi((\tau^{>0}\cX)[1])\ar[d]\\
&\phi(\cX^{-1})\ar[r]^{\phi(d^{-1})}&\phi(\cX^{0})\ar[r]^{\phi(d^0)}\ar[u]& \phi(\cX^{1})\ar[d]\\
&&\phi((\tau^{<0}\cX)[-1])\ar[u]&\phi((\tau^{>1}\cX)[2]).
}$$
We used notation as $\tau^{\le i}\cX$ to denote the naive truncations of the complex $\cX$.
The first row is exact, by the distinguished triangle $\tau^{>0}\cX\to\cX\to \tau^{\le 0}\cX$ and the fact that $\phi(\tau^{>0}\cX)=0$ by assumption (b). The two columns are similarly exact. By exactness and commutativity, it follows that the morphism from the kernel of $\phi(\cX^{0})\to\phi(\cX^{1})$ to $\phi(\tau^{\le 0}\cX)$ factors via $\phi(\cX)$. This morphism restricts to zero on the image of $\phi(\cX^{-1})\to\phi(\cX^{0})$ since the morphism $\cX^{-1}\to \tau^{\le 0}\cX$ is nullhomotopic (so zero in $K^b\ba$).
This yields a morphism $(\phi|_{\ba})^0_\Delta(\cX)\to \phi(\cX)$. Now consider a morphism $\cX\to\cY$ in $K^b\ba$. We choose some lift in the category of complexes (since $\tau^{\le 0}$ is not a functor on $K^b\ba$) which allows us to construct a commutative diagram combining the above diagram with the one for $\cY$. This demonstrates that the morphism $(\phi|_{\ba})^0_\Delta(\cX)\to \phi(\cX)$ is natural in $\cX$.

That this natural transformation is an isomorphism can be proved again by induction on the length of complexes and the 5-lemma.
\end{proof}

\subsection{Multi-representability}

In this section we define homological topologies on $\ba$ simply as Grothendieck topologies on $\Noy\ba$. The justification for introducing a new term stems from the variety of alternative realisations of this set in Theorem~\ref{ThmAlternative} below.
\begin{definition}${}$\begin{enumerate}
\item
The set of {\bf homological topologies on $\ba$} is $\HTop(\ba):=\Top(\Noy\ba)$.
\item For a functor $\theta:\ba\to\bC$ to an AB5 category $\bC$, its homological topology $\HT_{\theta}\in\HTop(\ba)$ is given by $\cT_{\vec\theta}\in\Top(\Noy\ba)$.
\item For a homological topology $\cR\in\HTop(\ba)$, we write
$$\tHZ:=\tZ\circ\tN\,:\, \ba\to\Sh(\Noy\ba,\cR)=:\HSh(\ba,\cR).$$
\end{enumerate}
\end{definition}

Note that, since $\vec\theta$ is prexact, see Theorem~\ref{Thm3Way}, the notation $\cT_{\vec\theta}$ from \ref{Ttheta} is justified.

Now we extend Theorem~\ref{PropFlatRep} from prexact functors to all additive functors. Denote by $\ABf$ the 2-full 2-subcategory of $\ABt$, from~\ref{Def2CatGro}, consisting of AB5 categories. 

\begin{theorem}\label{ThmUniHSh}
\begin{enumerate}
\item The 2-functor $[\ba,-]:\ABf\to\Cat$ decomposes as $$[\ba,-]=\coprod_{\cR\in\HTop(\ba)}[\ba,-]_{\cR},$$
where $[\ba,\bC]_{\cR}$ is the category of functors $\theta:\ba\to\bC$ with $\HT_\theta=\cR$.
\item For each $\cR\in\HTop(\ba)$, the 2-functor 
$$[\ba,-]_{\cR}:\;\ABf\to\Cat$$
is represented by $\HSh(\ba,\cR)$. More concretely, for each $\bC\in\ABf$ composition with $\tHZ$ yields an equivalence
$$\ABf(\HSh(\ba,\cR),\bC)\;\stackrel{\sim}{\to}\; [\ba,\bC]_{\cR}.$$
\item For every $\cR\in\HTop(\ba)$, the functor $\tHZ:\ba\to\HSh(\ba,\cR)$ is AB3-tight.
\end{enumerate}

\end{theorem}
\begin{proof}
For $F\in\ABf(\bC,\bB)$ and $\theta\in[\ba,\bC]$, we have $F\circ\vec{\theta}=\overrightarrow{F\circ \theta}$, and hence $\HT_\theta=\HT_{F\circ \theta}$ by Example~\ref{ExTopPrex}(2). This proves part (1).

By definition and Theorem~\ref{Thm3Way}(1), composition with $\tN$ yields an equivalence
$$[\Noy\ba,\bC]_{prex,\cR}\;\xrightarrow{\sim}\;[\ba,\bC]_{\cR}.$$
The proof of (2) is then concluded by applying Theorem~\ref{PropFlatRep} to $\Noy\ba$.

For part (3), consider an AB3 subcategory $\bA\subset\Sh(\Noy\ba,\cR)$ which contains the image of $\tHZ$. Since $\bA$ is is closed under kernels and since $\vec{\tHZ}:\Noy\ba\to \Sh(\Noy\ba,\cR)$ is isomorphic to $\tZ$ by Theorem~\ref{Thm3Way}(1), it follows that $\bA$ contains the image of $\tZ$. Since every object in $\Sh(\Noy\ba,\cR)$ is a cokernel of a morphism between direct sums of objects in the image of $\tZ$, and $\bA$ is closed under such operations, it follows indeed that $\bA=\Sh(\Noy\ba,\cR)$.
\end{proof}

To state the following theorem, we define a special type of sieves in $\Noy\ba$. For every morphism $f:X^0\to X^1$ in $\ba$ viewed as an object in $\Noy\ba$, we let $R_f\subset \Noy\ba(-,f)$ be the sieve generated by all morphisms 
$$\xymatrix{
Y\ar[rr]\ar[d]^-{a}&& 0\\
X^0\ar[rr]^f &&X^1
}$$
in $\Noy\ba$ to $f$ from objects in the image of $\tN :\ba\to\Noy\ba$.
Recall that by definition such morphisms  satisfy $f\circ a=0$.

\begin{theorem}\label{ThmIota}
There exists an injection $\iota:\Top(\ba)\hookrightarrow\HTop(\ba)$ with the following properties.
\begin{enumerate}
\item For each $\cT\in\Top(\ba)$, there is an equivalence 
$$\Sh(\ba,\cT)\;\simeq\;\HSh(\ba,\iota(\cT))$$
which yields a commutative triangle with the functors $\tZ$ and $\tHZ$ out of $\ba$.
\item For a prexact functor $\theta:\ba\to\bC$, we have $\iota(\cT_\theta)=\HT_{\theta}$.
\item The image of $\iota$ consists of the Grothendieck topologies $\cR$ on $\Noy\ba$ for which $R_f\in \cR(f)$, for each object $f:X^0\to X^1$ in $\Noy\ba$.
\item The partial function $u_t:\Top(\bb)\rightharpoonup\Top(\ba)$ from~\ref{DefPartFun} induced by a functor $u:\ba\to\bb$ is obtained from restricting source and target along $\iota$ of the function
$$\HTop(\bb)\,\to\,\HTop(\ba),\;\cR\mapsto \HT_{\tHZ\circ u}\quad\mbox{for }\; \tHZ:\bb\to\HSh(\bb,\cR).$$ 
\end{enumerate}
\end{theorem}
\begin{proof}
We define a function $\iota:\Top(\ba)\to\HTop(\ba)$ which sends $\cT$ to $\HT_{\tZ}$ for $\tZ:\ba\to\Sh(\ba,\cT)$.

By Theorem~\ref{ThmUniHSh}(2), we find a commutative diagram
$$\xymatrix{
&&\HSh(\ba,\iota(\cT))\ar[d]\\
\ba\ar[rr]^{\tZ}\ar[rru]^{\tHZ} &&\Sh(\ba,\cT),
}$$
where the vertical arrow is exact, faithful and cocontinuous. It follows (Examples~\ref{ExPrex}(2) and~\ref{ExTopPrex}(2)) that $\tHZ$ is prexact with $\cT_{\tHZ}=\cT$. By Theorem~\ref{PropFlatRep}, there is also an upwards exact faithful cocontinuous functor yielding a commutative diagram. The universalities in Theorems~\ref{PropFlatRep} and~\ref{ThmUniHSh} show that the two functors are mutually inverse, proving (1). Part (2) follows from the same universality.


Similarly, we can observe that when $\tHZ$ is prexact, then it is in the image of $\iota$. To prove (3), it thus suffices to show that $\tHZ:\ba\to\HSh(\ba,\cR)$ is prexact if and only if $\cR\in\HTop(\ba)$ satisfies the condition in (3). By Theorem~\ref{Thm3Way}(1), $\vec{\tHZ}\simeq\tZ_N$, where we write $\tZ_N:\Noy\ba\to\Sh(\Noy\ba,\cR)$ in order to avoid confusion with the $\tZ$ in the previous paragraph. Application of Theorem~\ref{ThmFlat} therefore shows that $\tHZ=\tZ_N\circ \tN$ is prexact if and only if for each object $f:Y\to X$ in $\Noy\ba$, the morphism
$$\bigoplus_{g:Z\to Y, f\circ g=0}\tZ_N(\tN(Z))\;\to\; \tZ_N(f)$$
is an epimorphism, which is the same as saying $R_f\in \cR(f)=\cT_{\tZ_N}$.

Part (4) follows from definition and part (1).
\end{proof}

\begin{remark}
The partial inverse to $\iota:\Top(\ba)\hookrightarrow\HTop(\ba)$ is given by the partial function $\tN_t:\Top(\Noy\ba)\rightharpoonup \Top(\ba)$ from \ref{DefPartFun}.
\end{remark}

\subsubsection{Extended example}\label{ExDN}We work out the universal Grothendieck categories predicted by Theorem~\ref{ThmUniHSh} in a specific example. Let $k$ be a field, consider the algebra of dual numbers $R=k[x]/x^2$ and set $\ba=R\mbox{-free}$. By Example~\ref{ExNoyA}, we know that $\Noy\ba$ is equivalent to the category of finite-dimensional $R$-modules. We have two indecomposable $R$-modules up to isomorphism: the projective module $P\simeq R$ and the simple module $L\simeq k$. We can classify Grothendieck topologies on $\Noy\ba$ directly. There turn out to be 4 possibilities for a topolgy $\cR$ on $R$-mod:
\begin{enumerate}
\item[(0)] The discrete topology $\cR_0$: $0\in\cR(X) $ for all $X$.
\item[(1)] The minimal sieve in $\cR_1$ on $P$ is generated by $L\hookrightarrow P$ and the only sieve on $L$ in $\cR_1$ is the representable functor itself.
\item[(2)] The minimal sieve in $\cR_2$ on $L$ is generated by  $P\tto L$ and the only sieve on $P$ in $\cR_2$ is the representable functor itself.
\item[(3)] The trivial topology $\cR_3$: the only sieves in $\cR_3$ are the representable functors.
\end{enumerate}
By Theorem~\ref{ThmIota}(3), the image of $\iota$ is $\{\cR_0,\cR_2\}$, so in particular, $\HSh(\ba,\cR_2)\simeq R\Mod$ by Theorem~\ref{ThmIota}(1).
Obviously $\Sh(\Noy\ba,\cR_0)=0$ and this category corresponds to zero functors out of $\ba$. A non-zero $\theta:\ba\to\bC$ induces homological topology $\cR_1$ if and only if the $R\xrightarrow{x}R$ is sent to zero. The universal property in Theorem~\ref{ThmUniHSh} thus shows that $\Sh(\Noy\ba,\cR_1)$ is the category of $k$-vector spaces. So $\Sh(\Noy\ba,\cR_0)$ and $\Sh(\Noy\ba,\cR_1)$ classify non-faithful functors and we summarise the results from the theory on faithful functors in the following proposition.

\begin{prop}\label{PropEx}
Let $\bC$ be a $k$-linear AB5 category with an object $X\in \bC$ and non-zero differential $d:X\to X$ (i.e. $d^2=0$). We consider the sequence $$X\xrightarrow{d}X\xrightarrow{d}X.$$
\begin{enumerate}
\item If the sequence
is acyclic there exists an essentially unique faithful exact cocontinuous functor $R\mbox{\rm-Mod}\to\bC$ with $R\mapsto X$ and $x\mapsto d$.
\item If the sequence
has non-zero homology there exists an essentially unique faithful exact cocontinuous functor $\PSh(R\mbox{\rm-mod})\to\bC$ with $\Yon(R)\mapsto X$ and $\Yon(x)\mapsto d$.
\end{enumerate}
\end{prop}
\begin{proof}
The obvious universal property of $\ba$ shows that we have an essentially unique (faithful) functor $\theta:\ba\to\bC$ with $\theta(x)=d$. The morphism $P\tto L$ in $\Noy\ba$ is sent to an epimorphism by $\vec{\theta}$ precisely when the sequence in the proposition is exact. The results are therefore immediate applications of Theorem~\ref{ThmUniHSh}. 
\end{proof}

\begin{remark}\label{RemDualN}
\begin{enumerate}
\item The universal category in \ref{PropEx}(2) is the category of modules over the path algebra of the quiver
$$\xymatrix{
e\ar@/^/[r]^a& s\ar@/^/[l]^b
}$$with relation $b\circ a=0$. In particular, if $k=\mC$ the category is equivalent to  $\Ind\mathcal{O}_0(sl_2)$, the ind-completion of the principal block in BGG category O of $sl_2(\mC)$. 
\item For all $i$, the functor $\ba\to\Sh(\Noy\ba, \cR_i)$ is full and for $i>1$ it is faithful. 
\end{enumerate}
\end{remark}

\subsection{Alternative realisations}
In this section we establish three alternative realisations of the universal category $\HSh(\ba,\cR)$.

\subsubsection{}
Consider the following two sets of Grothendieck topologies:
\begin{itemize}
\item The set $\Top(K_+^b\ba)_0$ of topologies $\cR$ on $K^b_+\ba$ with $0\in\cR(X[i])$ for all $i<0$ and $X\in\ba$.
\item The set $\Top(K^b\ba)_0$ of topologies $\cR$ on $K^b\ba$ with $0\in\cR(X[i])$ for all $i\not=0$ and $X\in\ba$.
\end{itemize}
\label{Rhat} For $\cR\in\Top(K^b\ba)_0$, we consider the following covering system $\hat{\cR}$ on $K^b\ba$. For each $\cX\in K^b\ba$, a sieve $S$ on $\cX$ is in $\hat{\cR}$ if and only if
$$S[i]\in\cR(\cX[i]),\quad\mbox{for all $i\in\mZ$}.$$
Since $\hat{\cR}$ is the intersection of a family of Grothendieck topologies (a countable number of `shifts' of the topology $\cR$), it is again a topology on $K^b\ba$.

\begin{theorem}\label{ThmAlternative}

\begin{enumerate}
\item There exist bijections
$$\xymatrix{ \Top(K_+^b\ba)_0&\HTop(\ba)\ar[r]^-{1:1}_-{\rho_2}\ar[l]_-{1:1}^-{\rho_1}&\Top(K^b\ba)_0}$$
such that, for each $\cR\in\HTop(\ba)$, we have a commutative diagram
$$\xymatrix{\Noy\ba\ar[d]&K_+^b\ba\ar[l]\ar[r]\ar[d]&K^b\ba\ar[d]\\
\Sh(\Noy\ba,\cR)&\Sh(K^b_+\ba,\rho_1(\cR))\ar[l]\ar[r]&\Sh(K^b\ba,\rho_2(\cR))}$$
where the bottom line consists of equivalences.
\item For $\cR\in\Top(K^b\ba)_0$, the category $\Sh(K^b\ba,\cR)$ is equivalent to the minimal AB3 subcategory of $\Sh(K^b\ba,\hat{\cR})$ which contains the image of $\ba\to \Sh(K^b\ba,\hat{\cR})$.
\item For an additive functor $\theta:\ba\to\bC$ to $\bC\in\ABf$, the homological topology $\HT_\theta$ is determined by the topology of $\theta^{\mZ}_\Delta: K^b\ba\to\bC^{\mZ}$ and vice versa.
\end{enumerate}
\end{theorem}

\begin{remark}
We will henceforth refer to an element of any of the sets $\Top(\Noy\ba)$, $\Top(K_+^b\ba)_0$ or $\Top(K^b\ba)_0$ as a `homological topology on $\ba$'. Furthermore, we will interpret $\HSh(\ba,\cR)$ accordingly as a sheaf category on $\Noy\ba$, $K_+^b\ba$ or $K^b\ba$ depending on what form is more convenient.
\end{remark}

We start the proof with the following general recognition result.

\begin{lemma}\label{LemRec}
Consider a homological Grothendieck topology $\cR$ on $\ba$ and $\phi\in[\ba,\bA]_{\cR}$ for some AB5 category $\bA$ for which $\phi:\ba\to\bA$ is AB3-tight.
If $-\circ\phi:\ABf(\bA,\bC)\to[\ba,\bC]_{\cR}$ is essentially surjective for all $\bC\in\ABf$, then $\phi$ is isomorphic to $\tHZ:\ba\to\HSh(\ba,\cR)$.
\end{lemma}
\begin{proof}
We claim there exists faithful exact cocontinuous functors $F,G$ for which the diagram
$$\xymatrix{
&\ba\ar[ld]_{\tHZ}\ar[d]^{\phi}\ar[rd]^{\tHZ}&\\
\HSh(\ba,\cR)\ar[r]^-{F}&\bA\ar[r]^-{G}&\HSh(\ba,\cR)
}$$
is commutative. Indeed, $G$ exists by assumption. Existence of $F$ follows from Theorem~\ref{ThmUniHSh}(2). The latter result also implies that $G\circ F$ is isomorphic to the identity functor. From this it follows that $F$ is also full. That $F$ is then an equivalence follows from Lemma~\ref{LemTight}.
\end{proof}

\begin{proof}[Proof of Theorem~\ref{ThmAlternative}]
The sets $\Top(K^b_+\ba)_0$ and $\Top(K^b\ba)_0$ comprise precisely the topologies of the prexact functors $\phi$ which satisfy conditions (2)(a) and (3)(a) in Theorem~\ref{Thm3Way}. We can therefore repeat the proof of Theorem~\ref{ThmUniHSh} to state multirepresentability of $[\ba,-]$ in terms of sheaf categories on $K^b_+\ba$ or $K^b\ba$ with respect to the above topologies. This shows that there must exist bijections between the three sets of topologies such that the bijection results in equivalences between the sheaf categories. We then also get commutativity of the desired diagram up to composition with $\ba\to K_+^b\ba$. That this lifts to commutativity of the original diagram follows from Theorem~\ref{Thm3Way}(2).


Now we prove part (2). First we observe that, for $\theta:\ba\to\bC$, if $\cR$ is the topology corresponding to a prexact functor $\theta_\Delta^0$, then $\hat{\cR}$ is the topology corresponding to the prexact functor $\theta_{\Delta}^{\mZ}: K^b\ba\to\bC^{\mZ}$.
 Let $\bA$ denote the minimal AB3 subcategory of $\Sh(K^b\ba,\hat{\cR})$ which contains the image of $\ba\to \Sh(K^b\ba,\hat{\cR})$. For every $\bC\in\ABf$, we have functors
 $$\xymatrix{[\ba,\bC]_{\cR}\ar[rr]^-{\theta\mapsto \theta_\Delta^{\mZ}}&&[K^b\ba,\bC^{\mZ}]_{prex,\hat{\cR}}&\ABf(\Sh(K^b\ba,\hat{\cR}),\bC^{\mZ})\ar[r]\ar[l]^-{\sim}_-{-\circ\tZ}&\ABf(\bA,\bC^{\mZ})},$$
 where the equivalence is from Theorem~\ref{PropFlatRep} and the right arrow is composition with the inclusion of $\bA$. 
 
 Now take $\theta\in [\ba,\bC]_{\cR}$ and consider the resulting functor $\theta'\in\ABf(\bA,\bC^{\mZ})$. The composite of $\theta'$ and $\ba\to\bA$ is isomorphic to $\ba\xrightarrow{\theta}\bC\hookrightarrow \bC^{\mZ}$. Tightness of $\ba\to\bA$ then shows that also $\bA\to\bC^{\mZ}$ actually takes values in $\bC$, so henceforth we interpret $\theta'\in \ABf(\bA,\bC)$. The composite of $\theta'$ and $\ba\to\bA$ is then just isomorphic to $\theta$. In particular, $\ba\to\bA$ induces homological topology $\cR$.
Now we can apply Lemma~\ref{LemRec}, since the functor $\ABf(\bA,\bC)\to[\ba,\bC]_{\cR}$ will send $\theta'$ to $\theta$, and therefore be essentially surjective.

For part (3), by (1), we can replace $\HT_\theta$ with the topology $\cR$ of the prexact functor $\theta_\Delta^0$. As discussed above, the topology of $\theta_\Delta^{\mZ}$ is then $\hat{\cR}$ (which is determined by $\cR$). The other direction is a consequence of part (2).
\end{proof}

\begin{remark}\label{RemSerre}
Denote by $\Ab(\ba)$ Freyd's universal abelian category associated to $\ba$, see \cite{Be, Freyd1}, which is the (abelian) subcategory of compact objects in $\PSh(\Noy\ba)$, so $\Ind\Ab(\ba)\simeq \PSh(\Noy\ba)$.
The set of Serre subcategories in $\Ab(\ba)$ is a subset of $\HTop(\ba)$. More precisely, assume that $\HSh(\ba,\cR)$ is of the form $\Ind\bA$ for a small abelian category $\bA$ such that $\tHZ:\ba\to\Ind\bA$ takes values in $\bA$. Then $\bA=\Ab(\ba)/I$ for a Serre subcategory $I$. Moreover, $\Ind(\Ab(\ba)/I)$ is always of the form $\HSh(\ba,\cR)$. The corresponding map $I\mapsto \cR$ realises the inclusion. In some specific cases, this inclusion is an equality, see for instance~\ref{ExDN}.
\end{remark}

\subsection{Homological kernel}
We develop some notions to describe the kernel of $\vec{\theta}$ efficiently.
Fix $A\in\ba$.
\begin{definition}\label{DefHK1}
The {\bf canonical kernel} at $A$ is the following sieve $\Sigma^{\ba}_A$ on $\tN A\in\Noy\ba$. For $f\in\Noy\ba$, $\Sigma^{\ba}_A(f)$ consists of those morphisms $\alpha: f\to \tN A$ which compose to zero with all morphisms $\tN Y\to f$ for all $Y\in\ba$; so
$$\Sigma^{\ba}_A(f)\;=\; \ker\left(\Noy\ba(f, \tN A)\,\to\, \prod_{g:\tN Y\to f}\ba(Y,A)\right).$$
\end{definition}

For a morphism $f:X^0\to X^1$ in $\ba$, by \eqref{EqNAf} we can also write the abelian group $\Sigma^{\ba}_A(f)$ without reference to $\Noy\ba$ as
$$\Sigma^{\ba}_A(f)\;=\; H_0\left(\ba(X^1,A)\xrightarrow{h\mapsto h\circ f}\ba(X^0,A)\xrightarrow{m\to (m\circ g)_g}\prod_{g: Y\to X^0, f\circ g=0}\ba(Y,A)\right).$$


\begin{remark}
For a fixed $f\in\Noy\ba$, we have the following extremal cases.
\begin{enumerate}
\item The collection of morphisms $\{\tN Y\to f\,|\, Y\in\ba\}$ is jointly epimorphic in $\Noy\ba$ if and only if $\Sigma^{\ba}_A(f)=0$ for all $A\in\ba$.
\item As a morphism in $\ba$, $f$ is monic if and only if  $\Sigma^{\ba}_A(f)=\Noy\ba(f,\tN A)$ for all $A\in\ba$.
\end{enumerate}\end{remark}
\begin{example}\label{ExStr}
For a ring $R$ we set $\ba=R\free$. For $r\in R$ interpreted as the morphism $R\xrightarrow{\cdot r}R$, we find
$$\Sigma^{\ba}_R(r)\;=\; \{u\in R\,|\, tu=0 \mbox{ for all $t\in R$ with } tr=0\}/rR.$$
More concretely:
\begin{enumerate}
\item If $R$ is a domain and $r\not=0$, then $\Sigma^{\ba}_R(r)=R/rR$.
\item If $R=k[x]/(x^2)$ for a field $k$, then $\Sigma^{\ba}_R(r)=0$ for all $r\in R$.
\end{enumerate}
\end{example}

\begin{example}
If $\ba$ is an abelian category and $I\in \ba$ injective, then $\Sigma^{\ba}_I=0$.
\end{example}


Now we fix an additive functor $\theta:\ba\to\bC$ to an abelian category $\bC$.

\begin{definition}
The {\bf homological kernel at $A$ of $\theta$} is the following sieve $\Sigma^{\ba}_A\theta$ on $\tN A\in\Noy\ba$. For $f\in\Noy\ba$, the subgroup $\Sigma^{\ba}_A\theta(f)\subset\Noy\ba(f,\tN A)$ consists of the morphisms in the kernel of $\vec{\theta}$, so
\begin{eqnarray*}\Sigma^{\ba}_A\theta(f)&=& \ker\left(\Noy\ba(f, \tN A)\,\to\, \bC(\vec{\theta}f,\theta A)\right)\\
&=&H_0\left(\ba(X^1,A)\to\ba(X^0,A)\to\bC(\ker \theta(f),\theta A)\right).\end{eqnarray*}
\end{definition}
The following proposition demonstrates why $\Sigma^{\ba}_A$ is the `canonical kernel', that this canonical kernel is achieved by faithful prexact functors and that $\{\Sigma^{\ba}_A\theta\,|\,A\in\ba\}$ contains enough information to describe $\ker\vec{\theta}$ (whence the name homological kernel).
\begin{prop}\label{PropReach}\label{ReachKer}
\begin{enumerate}
\item There is an inclusion
$\Sigma^{\ba}_A\theta\subset \Sigma^{\ba}_A$ of sieves on $\tN A\in\Noy\ba$ if and only if $\theta:\ba(-,A)\to\bC(-,\theta A)$ is injective. 
\item If $\bC$ is an AB5 category and $\theta$ is prexact, then there is an inclusion
$\Sigma^{\ba}_A\subset \Sigma^{\ba}_A\theta$ of sieves on $\tN A\in\Noy\ba$.
\item If $\bC$ is an AB5 category and $\theta$ is faithful and prexact, then
$\Sigma^{\ba}_A= \Sigma^{\ba}_A\theta$.

\item We have $\Sigma^{\ba}_A=\Sigma^{\ba}_A\Yon$, for the Yoneda embedding $\Yon:\ba\to\PSh\ba$.
\item A morphism $ f\to g$ in $\Noy\ba$, represented by $\alpha:X^0\to Y^0$ as in equation~\eqref{ExNoy}, is in the kernel of $\vec{\theta}$ if and only if the class of $\alpha$ in $\Noy\ba(f,\tN Y^0)$ is in $\Sigma^{\ba}_{Y^0}\theta(f)$.
\end{enumerate}
\end{prop}
\begin{proof}Fix an object $f:X^0\to X^1$ in $\Noy\ba$ and
consider the commutative diagram
$$\xymatrix{
&&\prod_{g: Y\to X^0, f\circ g=0}\ba(Y,A)\ar[rd]\\
\ba(X^1,A)\ar[r]&\ba(X^0,A)\ar[ur]\ar[rd]&&\prod_{g: Y\to X^0, f\circ g=0}\bC(\theta Y,\theta A).\\
&&\bC(\ker\theta(f),\theta A)\ar[ur]
}$$
The morphisms in the left-hand side are the ones defining $\Sigma_A^{\ba}(f)$ and $\Sigma_A^{\ba}\theta(f)$. The rightmost downwards arrow is given by $\theta$.
The rightmost upwards arrow is constructed using the universality of the kernel. Concretely, for each $g$, it follows that $\theta(g)$ factors via a unique morphism $\theta Y\to \ker\theta(f)$, and the arrow is composition with said morphism in each factor.

If $\bC$ is an AB5 category and $\theta$ is prexact, then by Theorem~\ref{ThmFlat}, the latter arrow is injective. This proves part (2). Conversely, if $\theta$ is faithful on $\ba(-,A)$, then the rightmost downwards arrow is injective, showing that $\Sigma^{\ba}_A\theta\subset \Sigma^{\ba}_A$ in part (1).

To prove the other direction of part (1) we can consider $f':B\to 0$ in $\Noy\ba$, since
$$\Sigma_A^{\ba}(\tN B)=0,\quad\mbox{and}\qquad \Sigma_A^{\ba}\theta(\tN B)=\ker (\ba(B,A)\to \bC(\theta B,\theta A)).$$

 Part (3) is a consequence of parts (1) and (2).
Part (4) follows immediately from the definitions. Alternatively it follows form part (3), as $\Yon$ is faithful and by definition prexact. Conversely we can also use (4) to prove (2) immediately from the definitions.
Since $\ker \theta(f)\to \ker \theta(g)$ is zero if and only if $\ker \theta(f)\to \theta Y^0$ is zero, part (5) follows.
\end{proof}

\begin{remark}\label{RemDefSigma}
\begin{enumerate}
\item The converse to Proposition~\ref{PropReach}(3) is not true. The property $\Sigma^{\ba}_A\theta=\Sigma^{\ba}_A$ for all $A\in\ba$ for a functor $\theta:\ba\to\bC$ to an AB5 category $\bC$ is not sufficient to conclude that $\theta$ is prexact. Indeed, by Proposition~\ref{PropReach} for $\ba$ as in Example~\ref{ExStr}(2), a functor $\ba\to\bC$ satisfies $\Sigma^{\ba}_R\theta=\Sigma^{\ba}_R$ if and only if it is faithful. By \ref{ExDN} not all such faithful functors are prexact.
\item If $F:\bC\to\bB$ is a faithful exact functor between abelian categories, then $\Sigma^{\ba}_A\theta=\Sigma^{\ba}_A(F\circ \theta)$.
\end{enumerate}
\end{remark}

The homological kernel of a faithful functor is not always canonical, as the following examples show.
\begin{example}\label{ExNotMax}
\begin{enumerate}
\item Consider a ring $S$ with subring $R$. Set $\ba=R\free$ and $\bC=S\Mod$ and let $\theta:\ba\to\bC$ be the functor $S\otimes_R-$. For $r\in R$, seen as the morphism $R\xrightarrow{\cdot r} R$, we find
$$\Sigma^{\ba}_R\theta(r)\;=\; \{u\in R\,|\, zu=0 \mbox{ for all $z\in S$ with } zr=0\}/rR,$$
which is in general a proper subgroup of $\Sigma^{\ba}_R(r)$ from Example~\ref{ExStr}. 
\item If $f:X^0\to X^1$ is monic in $\ba$ and $\theta:\ba\to\bC$ faithful, then $\Sigma^{\ba}_A\theta(f)=\Sigma^{\ba}_A(f)$ if and only if for every $\alpha\in\ba(X^0,A)$, the morphism $\theta(\alpha)$ restricts to zero on $\ker\theta(f)$. Examples where is the latter condition is not satisfied will be given in Section~\ref{SecMono}.
\end{enumerate}
\end{example}


\section{Comments on monoidal and enriched versions}\label{SecMon}

Let $(\ba,\otimes,\unit)$ be a preadditive monoidal category and $(\bC,\otimes,\unit)$ an abelian monoidal category. We assume that both have additive tensor product $-\otimes-$.

\subsection{Multi-representability}

\subsubsection{} Assume that $\bC$ is an AB3 category for which the tensor product is cocontinuous in each variable. The 2-category of such monoidal categories, with 1-morphisms given by faithful exact cocontinuous monoidal functors and 2-morphisms natural transformations of monoidal functors is denoted by $\ABt^{\otimes}$.

For a monoidal prexact functor $\theta:\ba\to\bC$, it follows from \cite[Theorem~3.5.5(c)]{PreTop} that the Grothendieck topology $\cT_\theta$ on $\ba$ is `monoidal' as defined {\it loc. cit.}  Consequently, $\Sh(\ba,\cT_\theta)$ is a biclosed monoidal Grothendieck category (the functors $N\otimes-$ and $-\otimes N$ have right adjoints for every object $N$) and $\tZ:\ba\to\Sh(\ba,\cT_\theta)$ is monoidal. 

It follows that the 2-functor 
$$[\ba,-]^\otimes_{prex}:\;\ABt^\otimes\to\Cat$$
decomposes into the 2-functors $[\ba,-]^\otimes_{prex,\cT}$, where $\cT$ runs over all monoidal Grothendieck topologies.
\begin{prop}\label{PropMon1}
For each monoidal Grothendieck topology $\cT$ on $\ba$, the 2-functor $$ [\ba,-]_{prex,\cT}^\otimes:\;\ABt^\otimes\to\Cat$$ is represented by the biclosed monoidal Grothendieck category $\Sh(\ba,\cT)$. More concretely, composition with $\tZ:\ba\to\Sh(\ba,\cT)$ yields an equivalence
$$\ABt^\otimes(\Sh(\ba,\cT),\bC)\;\stackrel{\sim}{\to}\; [\ba,\bC]_{prex,\cT}^\otimes.$$
\end{prop}
\begin{proof}
By \cite[Theorem~3.5.6]{PreTop}, and with notation and results from the proof of Theorem~\ref{PropFlatRep}, composition with $\tZ$ yields an equivalence
$$[\Sh(\ba,\cT),\bC]_{cc}^\otimes\;\stackrel{\sim}{\to}\; [\ba,\bC]_{\cS}^\otimes.$$
The proof that this restricts to an equivalence between subcategories yielding the proposition is identical to the proof of Theorem~\ref{PropFlatRep}, as it only relates to the underlying functors and not their monoidal structure.
\end{proof}

\subsection{Monoidal structure on the kernel category}Assume $\ba$ is additive.

\subsubsection{}
The category $K^b\ba$ is monoidal (rigid when $\ba$ is rigid), with 
$$(\cX\otimes \cY)^i=\bigoplus_{j\in\mZ} \cX^{i-j}\otimes \cY^j.$$
The differential on the tensor product is given by
$$\cX^a\otimes\cY^b\;\xrightarrow{(d\otimes\cY^b,(-1)^{a}\cX^a\otimes d)}\;\cX^{a+1}\otimes\cY^b\;\oplus\;\cX^a\otimes\cY^{b+1}.$$
The functor $\ba\to K^b\ba$ is canonically monoidal. The subcategory $K^b_+\ba\subset K^b\ba$ is a monoidal subcategory. 

Also $\Noy\ba$ has a monoidal structure. The tensor product of $f:X^0\to X^1$ and $g:Y^0\to Y^1$ is given by
$$X^0\otimes Y^0\xrightarrow{(f\otimes Y^0,X^0\otimes g)}X^1\otimes Y^0\oplus X^0\otimes Y^1.$$
The functor $K_+^b\ba\to\Noy\ba$ is then also monoidal.

If $\ba$ is braided (or symmetric), then so are all the above monoidal categories and functors (where we use the Koszul sign rule).

\subsubsection{} Consider a monoidal functor $\theta:\ba\to\bC$. The left exact functor $\vec{\theta}$ inherits a lax monoidal structure, coming from the obvious morphisms
$$\ker\theta(f)\otimes\ker\theta(g)\;\to\; \ker\left(\theta( X^0\otimes Y^0\,\to\,  X^0\otimes Y^1 \,\oplus\,  X^1\otimes  Y^0)\right).$$
If the tensor product in $\bC$ is left exact, it follows that the lax monoidal structure on $\vec{\theta}$ is actually monoidal.

\begin{lemma}\label{LemPrexMon}
\begin{enumerate}
\item Composition with the monoidal functor $\tN :\ba\to\Noy\ba$ yields a fully faithful functor
$$[\Noy\ba,\bC]^{\otimes}_{lex}\;\to\; [\ba,\bC]^{\otimes}.$$
\item If the tensor product in $\bC$ is left exact, the functor in (1) is an equivalence.
\end{enumerate}
\end{lemma}
\begin{proof}
If $\phi\in [\Noy\ba,\bC]^{\otimes}_{lex}$, then the isomorphism between $\phi$ and $\overrightarrow{\phi\circ \tN }$ from the proof of Theorem~\ref{Thm3Way}(1) is a natural transformation of (lax) monoidal functors. The assignment $\theta\mapsto \vec{\theta}$ therefore provides an inverse to the functor in part (1), when restricted to the essential image of the latter functor. For part (2) it suffices to observe that the inverse is now defined everywhere.
\end{proof}

\subsubsection{} Now assume that the tensor product in $\bC$ is right exact. Then for two bounded chain complexes $\cX$ and $\cY$ there are canonical morphisms
$$\bigoplus_{a+b=i}H^a(\cX)\otimes H^b(\cY)\;\to\; H^{i}(\cX\otimes \cY).$$
These are isomorphisms if the tensor product is exact, by the K\"unneth theorem for chain complexes

 For a monoidal functor $\theta:\ba\to\bC$, the homological functors
\begin{equation}\label{hatThetaTensor}\theta^0_+:K_+^b\ba\to\bC\quad\mbox{and}\quad\theta^\mZ_\Delta: K^b\ba\,\to\, \bC^{\mZ}\end{equation}
are therefore canonically lax monoidal, where we view $\bC^{\mZ}$ as a monoidal subcategory of $K^b\bC$. If the tensor product on $\bC$ is exact, then $\theta^0_+$ and $\theta^\mZ_\Delta$ are actually monoidal.
Again, also the braiding of (lax) monoidal functors is carried over. 

\subsection{Homological kernels of monoidal functors}
Assume $\ba$ is additive and {\em rigid}.

\subsubsection{}\label{DefHK}
The {\bf canonical homological kernel} of $\ba$ is the sieve $\Sigma^{\ba}:=\Sigma^{\ba}_{\unit}$ on $\unit\in\Noy\ba$ from Definition~\ref{DefHK1}. Concretely, for a morphism $f:X^0\to X^1$ in $\ba$:
$$\Sigma^{\ba}(f)=H_0\left(\ba(X^1,\unit)\to\ba(X^0,\unit)\to\prod_{g: Y\to X^0, f\circ g=0}\ba(Y,\unit)\right).$$
Similarly, for a monoidal functor $\theta:\ba\to\bC$, the sieve $\Sigma^{\ba}\theta=\Sigma^{\ba}_\unit\theta$ on $\unit\in\Noy\ba$ given by
$$\Sigma^{\ba}\theta(f)= H_0\left(\ba(X^1,\unit)\to\ba(X^0,\unit)\to\bC(\ker \theta(f),\unit)\right), $$
is the {\bf homological kernel of $\theta$.} 

\begin{lemma}\label{LemSigmaMon}
For a monoidal functor $\theta:\ba\to\bC$,
\begin{enumerate}
\item $\Sigma^{\ba}\theta\subset\Sigma^{\ba}$ if and only if $\theta$ is faithful.
\item $\Sigma^{\ba}\theta\supset\Sigma^{\ba}$ when $\theta$ is prexact and $\bC$ is AB5.
\end{enumerate}
\end{lemma}
\begin{proof}
Part (1) follows from Proposition~\ref{PropReach}(1) and the fact that $\ba$ is rigid and $\theta$ monoidal. Part (2) is a special case of Proposition~\ref{PropReach}(2).
\end{proof}

For two posets $(\Lambda,\le)$, $(P,\preceq)$, a function $f:\Lambda\to P$ is an {\bf embedding} if
$$f(\lambda)\preceq f(\mu)\;\Leftrightarrow\;\lambda\le\mu$$
for all $\lambda,\mu\in\Lambda$. Clearly this forces $f$ to be injective.
\begin{prop}\label{BigProp1}
There exist inclusion preserving functions
$$\{\mbox{ideals in $\Noy\ba$}\}\;\xrightarrow{\mu}\;\{\mbox{sieves on $\unit\in\Noy\ba$}\}\;\xleftarrow{\nu}\; \{\mbox{ideals in $K^b\ba$}\}$$
such that, with $C$ the class of monoidal functors $\theta:\ba\to \bA$ to abelian monoidal categories $\bA$ with exact tensor product:
\begin{enumerate}
\item $\mu$ restricts to an embedding on the subset of ideals of the form $\ker\vec\theta$, for $\theta\in C$;
\item $\nu$ restricts to an embedding on the subset of ideals of the form $\ker\theta^{\mZ}_\Delta$, for $\theta\in C$;
\item we have $\mu(\ker\vec\theta)=\Sigma^{\ba}\theta=\nu(\ker\theta_\Delta^{\mZ})$ for all $\theta\in C$.
\end{enumerate}
\end{prop}
\begin{proof}
For an ideal $I$ in $\Noy\ba$, the sieve $\mu(I)$ is defined as the restriction $\mu(I)=I(-,\unit)$. 

For $f: X^0\to X^1\in\Noy\ba$, we have the associated complex in degrees 0,1
$$\cX_f:\,0\to X^0\to X^1\to 0\;\,\in \,K_+^b\ba\subset K^b\ba$$
and the functors in \ref{secdiag} yield isomorphisms
$$\Noy\ba(f,\unit)\;\xleftarrow{\sim}\;K_+^b\ba(\cX_f,\unit)\;\xrightarrow{\sim}\; K^b\ba(\cX_f,\unit).$$
The sieve $\nu(J)$ is defined by setting $\nu(J)(f)$ equal to the image of $J(\cX_f,\unit)$ under the isomorphism.


Now we prove (1). Consider an arbitrary morphism in $\Noy\ba$ represented by some $\alpha$ as in~\eqref{ExNoy}. It follows quickly from the definition of $\vec{\theta}$ that $\alpha$ is in the kernel of $\vec{\theta}$ if and only if the corresponding morphism $f\to \tN Y^0$ (the composition of $\alpha$ with $(Y^0\to Y^1)\to \tN Y^0$) is in the kernel, see also \ref{PropReach}(5). Since $\tN Y^0$ has a left dual, it follows that $\alpha$ is in the kernel of $\vec{\theta}$ if and only if the associated morphism
$$\xymatrix{
(Y^0)^\ast\otimes X^0\ar[rr]^-{(Y^0)^\ast\otimes f}\ar[d]&&(Y^0)^\ast\otimes X^1\\
\unit\ar[rr]&&0
}$$
is in the kernel of $\vec{\theta}$, or equivalently in $\mu(\ker\vec{\theta})$. So $\mu(\ker\vec\theta)$ retains all information in $\ker\vec\theta$.

Now we prove (2). Since $K^b\ba$ is rigid, the kernel of the monoidal functor $\theta_\Delta^\mZ$ is determined by which morphisms $(\cX,d)\to\unit$ are sent to zero. Using the definition of $\theta_\Delta^\mZ$, such a morphism is sent to zero if and only if the morphism $d^0\to\unit$ in $\Noy\ba$ is in $\nu(\ker\theta_\Delta^\mZ)$.

Part (3) follows from the definitions.
\end{proof}

\subsection{Enriched versions}

In all previous sections we considered $\mZ$-linear (preadditive) categories. All results can be straightforwardly adapted to $K$-linear categories for commutative rings $K$. Note that for a $K$-linear category $\ba$, $\PSh\ba$ is equivalent to the category of $K$-linear functors $\ba^{\op}\to K\Mod$ and $K$-linear Grothendieck topologies coincide with the additive topologies.

For example, if we define ${\ABt}^{\otimes}$ to be the 2-category of all monoidal $K$-linear AB3 categories with $K$-linear cocontinuous tensor product, let $\ba$ be $K$-linear monoidal with $K$-linear tensor product and reserve the notation $[-,-]$ for categories of $K$-linear functors, Proposition~\ref{PropMon1} remains valid as stated. In the remainder of the paper we will use such enriched analogues of previous results without further comment.


\vspace{2mm}
\part{Local abelian envelopes}\label{Part2}

For the remainder of the paper, we let $k$ be a field. All functors are assumed to be $k$-linear and correspondingly, from now on, the notation $[-,-]$ will only be used to denote categories of $k$-linear functors.

\section{Universal tensor functors of fixed homological kernel}\label{LocAbEnvTheory}
For the entire section, let $\ba$ be a rigid monoidal category over $k$ (see~\ref{Defk}). By a tensor category we will understand a tensor category over some field extension $K/k$ (including $K=k$). Unless further specified, $\bT$ is assumed to be such a tensor category.


\subsection{The tensor category associated to a prexact monoidal functor}

\begin{theorem}\label{FlatMon}For a prexact monoidal functor $\theta:\ba\to\bT$, there exists a tensor category $\bU$ with $\Ind\bU\simeq \Sh(\ba,\cT_\theta)$, over an intermediary field extension $k<L<K$, with a monoidal functor $\theta':\ba\to\bU$ which yields an equivalence
$$-\circ\theta':\;\Tens^{\uparrow}(\bU,\bT_1)\;\stackrel{\sim}{\to}\;\RMon^{\uparrow}(\ba/{\ker\theta},\bT_1)$$
for each tensor category $\bT_1$. Here $\cT_\theta$, so in particular also $\bU$, depends only on $\ker\theta$.
\end{theorem}
\begin{proof}
By Proposition~\ref{PropMon1} we have the biclosed Grothendieck category $\Sh(\ba,\cT_\theta)$ with the exact cocontinuous faithful monoidal functor $T:\Sh(\ba,\cT_\theta)\to\Ind\bT$ such that $T\circ \tZ\simeq\theta$ (where we leave out the inclusion functor $\bT\to\Ind\bT$). It follows from \cite[Proposition~3.3.2]{CEOP} that $\Sh(\ba,\cT_\theta)\simeq \Ind\bU$ for some tensor category $\bU$. The universal property of $\theta'$ then follows immediately from \cite[Theorem~4.4.1]{PreTop}. This universal property shows that $\bU$ only depends on $\ker\theta$. Alternatively, that $\cT_\theta$ depends only on $\ker\theta$ is in \cite[Theorem~4.4.1]{PreTop}.
\end{proof}

\begin{corollary}\label{CorFlatKernel}
For two prexact monoidal functors $\theta:\ba\to\bT_1$ and $\phi:\ba\to \bT_2$ to tensor categories, their kernels are either the same or incomparable.
\end{corollary}
\begin{proof}
By Theorem~\ref{FlatMon} we can replace $\theta$ and $\phi$ by functors ($\theta'$ and $\phi'$) which are creators in the terminology of \cite[4.1.2]{PreTop}. The claim then follows from \cite[Corollary~4.4.4]{PreTop}.
\end{proof}

\begin{corollary}\label{CorFlatAll}
If there exists a faithful and prexact monoidal functor $\theta:\ba\to\bT$, then every faithful monoidal functor from $\ba$ to a tensor category is prexact. If furthermore $\bT$ is a tensor category over $k$, then $\ba$ admits a weak abelian envelope.
\end{corollary}
\begin{proof}
Firstly we observe that $\theta'$ in Theorem~\ref{FlatMon} is prexact. This follows either by construction, or from Example~\ref{ExPrex}(2). That all faithful monoidal functors to tensor categories are prexact now follows from the equivalence in Theorem~\ref{FlatMon} and again Remark~\ref{ExPrex}(2).
\end{proof}

\begin{remark}
A direct construction of $\bU$ (without reference to $\cT_\theta$) is given in \cite[\S 4]{PreTop}, see for instance \cite[Remark~4.4.3]{PreTop}. To present one of the consequences, set $\bb:=\ba/{\ker\theta}$. The category $\bU$ is a tensor category over $k$ if for each non-zero $u:U\to\unit$ in $\bb$, the following sequence is exact:
$$0\to \bb(\unit,\unit)\xrightarrow{-\circ u} \bb(U,\unit)\xrightarrow{-\circ(u\otimes U-U\otimes u)}\bb(U\otimes U,\unit).$$
\end{remark}

\begin{example}\label{TrivEx}
Consider a fully faithful monoidal functor $\theta:\ba\to\bT$ such that every object in $\bT$ is a quotient of one in the image  (or, more generally, $\theta$ satisfies the conditions in \cite[Lemma~4.1.3]{PreTop}), then $\theta$ is prexact by Corollary~\ref{CorFlat}.
\end{example}

\begin{prop}\label{BigProp2}
 The following conditions are equivalent for monoidal functors $\theta,\phi$ from $\ba$ to tensor categories.
\begin{enumerate}[label=(\alph*)]
\item $\Sigma^{\ba}\theta=\Sigma^{\ba}\phi$;
\item $\ker\vec{\theta}=\ker\vec{\phi}$;
\item $\ker\theta_\Delta^\mZ=\ker\phi_\Delta^\mZ$;
\item $\HT_{\theta}=\HT_{\phi}$.
\end{enumerate}

\end{prop}
\begin{proof}
The equivalence between conditions (a), (b) and (c) is an immediate consequence of Proposition~\ref{BigProp1}. Now we prove that (c) implies (d). Since $K^b\ba$ is rigid and $\theta_\Delta^\mZ$ and $\phi_\Delta^\mZ$ are prexact, $\ker\theta_\Delta^\mZ=\ker\theta_\Delta^\mZ$ implies that the topologies of $\theta_\Delta^\mZ$ and $\phi_\Delta^\mZ$ coincide, by Theorem~\ref{FlatMon}. The latter then implies that $\HT_{\theta}=\HT_{\phi}$, by Theorem~\ref{ThmAlternative}(3).

That (d) implies (b) follows from application of Corollary~\ref{CorTriv} to $\vec\theta$ and $\vec\phi$.
\end{proof}

\begin{remark}
Similarly one can prove that the kernels of $\theta^{0}_\Delta$ and $\theta^{\mZ}_\Delta$ determine one another, extending the list of equivalent conditions in Proposition~\ref{BigProp2}.
\end{remark}

Theorem~\ref{FlatMon}, together with the following example, allows us to recover several results from \cite{BEO, AbEnv, EHS}. 
Recall that $\ba$ is {\bf self-splitting} if for every morphism $f$ in $\ba$, there exists $0\not=V\in \ba$ such that $V\otimes f$ is split. We say that a morphism is split if is a direct sum of an isomorphism and a zero morphism.  
\begin{example}\label{ExSelf}
If $\ba$ is self-splitting, then every monoidal functor $\theta:\ba\to\bT$ is prexact. Indeed, for $f:Y\to X$ in $\ba$, take $0\not=V\in \ba$ for which there is a split exact sequence 
$$0\to Z'\xrightarrow{g'} V\otimes Y\xrightarrow{V\otimes f} V\otimes X$$
in $\ba$. It follows easily that 
$$V^\ast\otimes Z'\xrightarrow{(\ev_V\otimes Y)\circ (V^\ast\otimes g')} Y\xrightarrow{f} X$$
is sent to an acyclic sequence under $\theta$, so we can apply Corollary~\ref{CorFlat}.
\end{example}

\subsection{Universal tensor categories}
\subsubsection{}

For a subsieve $\gamma\subset\Sigma^{\ba}$, see \ref{DefHK}, we denote by $\RMon^{\uparrow}_\gamma(\ba,\bT)$ the category of (faithful) monoidal functors with homological kernel $\gamma$, which is well defined by Lemma~\ref{LemSigmaMon}(1).
The set of subsieves $\gamma\subset\Sigma^{\ba}$ for which there exist such a faithful monoidal functor is denoted by $\Sw(\ba)$. In other words, the set  $\Sw(\ba)$ contains the `admissible homological kernels'. We also consider the subset $\Sw_0(\ba)\subset\Sw(\ba)$ of those homological kernels which can be realised using faithful monoidal functors to tensor categories over $k$.

\begin{theorem}\label{ThmUniTensor}
\begin{enumerate}
\item The 2-functor $\RMon^{\uparrow}(\ba,-):\Tens^{\uparrow}\to\Cat$ decomposes as 
$$\RMon^{\uparrow}(\ba,-)\;=\;\coprod_{\gamma\in\Sw(\ba)}\RMon^{\uparrow}_\gamma(\ba,-).$$
\item For each $\gamma\in\Sw(\ba)$, the 2-functor $\RMon^{\uparrow}_\gamma(\ba,-)$ is representable by a tensor category $\bT_\gamma$. Moreover $\Ind\bT_\gamma\simeq\HSh(\ba,\cR)$ for a monoidal Grothendieck topology $\cR$ on $\Noy\ba$.
\item There are no inclusions between the different subfunctors of $\Sigma^{\ba}$ in $\Sw(\ba)$.
\end{enumerate}
\end{theorem}

We start the proof with the following lemma.
\begin{lemma}\label{LemCEOP}
Consider a $k$-linear biclosed Grothendieck category $\bC$ with a faithful exact cocontinuous monoidal functor $\bC\to\Ind\bT$. If $\bC$ has a class $S$ of rigid objects such that every object in $\bC$ is a quotient of a coproduct of kernels between objects in $S$, then $\bC$ is itself the ind-completion of a tensor category.
\end{lemma}
\begin{proof}
The assumptions imply that the tensor product on $\bC$ is exact, from which it follows that the kernel of a morphism between two rigid objects is again rigid. Indeed, for $N$ the kernel of $f: X\to Y$ let $Q$ denote the cokernel of $f^\ast:Y^\ast\to X^\ast$. Using exactness we easily obtain morphisms
$$\unit\to N\otimes Q,\quad Q\otimes N\to \unit$$
from the coevalualtion and evaluation for $(X,X^\ast)$ which inherit the snake relations. We can thus replace $S$ with the class of kernels of morphisms between objects in $S$. Then we can just assume that every object in $\bC$ is simply a quotient of a coproduct of objects in $S$ and the lemma is now in \cite[Proposition~3.3.2]{CEOP}.
\end{proof}

\begin{proof}[Proof of Theorem~\ref{ThmUniTensor}]
Part (1) follows from Remark~\ref{RemDefSigma}(2).

Now take $\gamma\in \Sw(\ba)$. By assumption, there exists a faithful monoidal functor $\theta:\ba\to\bT$ with homological kernel $\gamma$. Set $\cR:=\HT_{\theta}$. By the combination of Lemma~\ref{LemPrexMon} and Proposition~\ref{PropMon1} (applied to $\Noy\ba$ rather than $\ba$), the corresponding biclosed Grothendieck category $\Sh(\Noy\ba,\cR)$ admits a faithful exact monoidal functor to $\Ind\bT$ and therefore is itself the ind-completion of a tensor category by Lemma~\ref{LemCEOP}. We let $\bT_\gamma$ be the latter tensor category. For an arbitrary tensor category $\bT_1$, we claim there exist equivalences
$$\Tens^{\uparrow}(\bT_\gamma,\bT_1)\simeq [\Ind\bT_\gamma,\Ind\bT]_{cc}^\otimes\simeq\ABt^\otimes(\Sh(\Noy\ba,\cR),\Ind\bT_1)\simeq [\Noy\ba,\Ind\bT_1]^{\otimes}_{prex,\cR},$$
which are all obtained from appropriate composition with the obvious functors.
The first equivalence is in the proof of \cite[Theorem~4.4.1]{PreTop}. For the second equivalence, we only need to argue that the functors in the second term are automatically faithful and exact. By the first equivalence we know that they are exact and faithful when restricted to $\bT_\gamma$, and that the extension to the ind-completion is still exact and faithful is then obvious. The third equivalence is Proposition~\ref{PropMon1}. The conclusion then follows from the equivalences
$$[\Noy\ba,\Ind\bT_1]^{\otimes}_{prex,\cR}\simeq[\ba,\Ind\bT_1]^{\otimes}_{\cR}\simeq \RMon_\gamma^{\uparrow}(\ba,\bT_1).$$
The first equivalence follows from Lemma~\ref{LemPrexMon}(2). As $\bT_1$ is equivalent to the category of rigid objects in $\Ind\bT_1$, see \cite[Lemma~1.3.7]{AbEnv}, the second equivalence follows from the equivalence of (a) and (d) in Proposition~\ref{BigProp2}.

Corollary~\ref{CorFlatKernel} shows that there is no inclusion between $\ker\theta_\Delta^{\mZ}$ and $\ker\phi_\Delta^{\mZ}$ for monoidal functors $\theta:\ba\to\bT$ and $\phi:\ba\to\bT'$ to tensor categories $\bT$ and $\bT'$. 
Part (3) then follows from Proposition~\ref{BigProp1}.
\end{proof}

\begin{example}\label{ExPrexHK}
Assume that there exists a prexact faithful monoidal functor $\ba\to\bT$. Then $\Sw(\ba)$ is a singleton, by Theorem~\ref{FlatMon}.
\end{example}
More generally, we have:

\begin{corollary}\label{CorAbEnv}
Assume there exists a (faithful) monoidal functor $\theta:\ba\to\bT$ of canonical homological kernel, $\Sigma^{\ba}\theta=\Sigma^{\ba}$. Then every faithful monoidal functor from $\ba$ to a tensor category is of canonical homological kernel. Moreover, there exists such a functor $\ba\to \bT_0$ which induces
an equivalence
$$\Tens^{\uparrow}(\bT_0,\bT_1)\;\xrightarrow{\sim}\;\RMon^{\uparrow}(\ba,\bT_1)$$
for every tensor category $\bT_1$.
\end{corollary}
\begin{proof}
By Theorem~\ref{ThmUniTensor}(3) and Lemma~\ref{LemSigmaMon}(1), every faithful monoidal functor from $\ba$ to a tensor category must be of canonical homological kernel.
The result is then an immediate application of Theorem~\ref{ThmUniTensor}(2).
\end{proof}

For completeness we formulate Theorem~\ref{ThmUniTensor} restricted to tensor categories over $k$.
\begin{theorem}\label{ThmUniTensork}
\begin{enumerate}
\item The 2-functor $\RMon(\ba,-):\Tens\to\Cat$ decomposes as 
$$\RMon(\ba,-)\;=\;\coprod_{\gamma\in\Sw_0(\ba)}\RMon_\gamma(\ba,-).$$
\item For each $\gamma\in\Sw_0(\ba)$, the 2-functor $\RMon_\gamma(\ba,-)$ is representable by a tensor category $\bT_\gamma$ over $k$.
\end{enumerate}
\end{theorem}

\begin{remark}
By Remark~\ref{RemSerre}, $\Sw(\ba)$ can also be interpreted as the set of Serre subcategories $I$ of $\Ab(\ba)$ for which $\ba\to\Ab(\ba)/I$ is faithful and $\Ab(\ba)/I$ is a tensor category, see \cite{BHP}, and thus realises the corresponding $\bT_\gamma$.
\end{remark}


\subsection{Recognising a local abelian envelope}
We give the following intrinsic, albeit tedious, characterisation of the universal functors in Theorem~\ref{ThmUniTensor}.
\begin{theorem}
Consider a faithful monoidal functor $\theta:\ba\to\bT$. Then $\theta$ is the universal functor of its homological kernel if and only if the following three conditions are satisfied:
\begin{enumerate}
\item[(G')] For every $X\in \bT$, there exists a morphism $f:A^0\to A^1$ in $\ba$ such that $X$ is a quotient of $\ker\theta(f)$.
\item[(F')] For every two morphisms $g:B^0\to B^1$ and $h:C^0\to C^1$ in $\ba$ and every morphism $\alpha :\ker\theta(g)\to \ker\theta(h)$, there exists morphisms $f: A^0\to A^1$, $u:A^0\to B^0$ and $v:A^0\to C^0$ in $\ba$ such that
\begin{enumerate}
\item $g\circ u$ and $h\circ v$ factor via $f$;
\item $\theta(u)$ restricts to an epimorphism $\ker\theta(f)\to\ker\theta(g)$;
\item the following square is commutative
$$\xymatrix{
\ker\theta(f)\ar@{^{(}->}[r]\ar[d]^{\alpha\circ\theta(u)}& \theta A^0\ar[d]^{\theta(v)}\\
\ker\theta(h)\ar@{^{(}->}[r]& \theta C^0.
}$$
\end{enumerate}
\item[(FF')] For every two morphisms $g:B^0\to B^1$ and $h:C^0\to C^1$ in $\ba$ and every morphism $\nu:B^0\to C^0$ for which $h\circ \nu$ factors via $g$ and $\theta(\nu)$ restricts to zero on $\ker\theta(g)$, there exists morphisms $f: A^0\to A^1$ and $\mu:A^0\to B^0$ such that
\begin{enumerate}
\item both $g\circ\mu$ and $\nu\circ\mu$ factor via $f$;
\item $\theta(\mu)$ restricts to an epimorphism $\ker\theta(f)\to\ker\theta(g)$.
\end{enumerate}
\end{enumerate}
\end{theorem}
\begin{proof}
By construction, $\theta:\ba\to\bT$ is one of the universal functors if $\vec{\theta}$ (composed with $\bT\hookrightarrow\Ind\bT$) corresponds to the canonical functor from $\Noy\ba$ to a category of sheaves on $\Noy\ba$. For the latter property, we can use the intrinsic description in \cite[Theorem~1.2]{Lowen}. Using the definition of $\vec{\theta}$ and the fact that it takes values in the category of compact objects in $\Ind\bT$ then yields the description.
\end{proof}
\subsection{Some considerations about the symmetric case}
\subsubsection{}\label{RemBraid}
All statements in Theorem~\ref{ThmUniTensor} and Corollary~\ref{CorAbEnv} remain true if we restrict to braided (or symmetric) categories and functors. For this we can apply \cite[Theorem~3.5.7]{PreTop}.

\begin{corollary}\label{CorSymm} Assume that $\ba$ is symmetric.
\begin{enumerate}
\item Consider field extensions $K_1/k$ and $K_2/k$ and faithful symmetric monoidal functors $\theta_i:\ba\to\Vecc_{K_i}$. Then $\theta_1$ and $\theta_2$ have the same homological kernel if and only if there exists a common field extension $K_1\hookrightarrow K\hookleftarrow K_2$ such that the two monoidal functors $(K\otimes_{K_i}-)\circ \theta_i:\ba\to \Vecc_K$ isomorphic.
\item Assume that $k$ is algebraically closed. Every two non-isomorphic faithful symmetric monoidal functors $\ba\to\Vecc_k$ have different (and hence incomparable) homological kernel. 

\end{enumerate}
\end{corollary}
\begin{proof}
For part (1), assume first that $\theta_1$ and $\theta_2$ have the same homological kernel $\gamma$. It follows from Theorem~\ref{ThmUniTensor}(2) that $\theta_1$ and $\theta_2$ are obtained by composition of $\ba\to\bT_\gamma$ with two tensor functors $\bT_\gamma\to\Vecc_{K_i}$. It follows from \cite[1.10-1.13]{Del90} that there exists a faithfully flat commutative $K_1\otimes_k K_2$-algebra $R$ such that the two resulting monoidal functors $\bT_\gamma\to \mathsf{Mod}_R$ are isomorphic. We take $K$ to  be the quotient of $R$ with respect to a maximal ideal. Again, the two resulting monoidal functors $\bT_\gamma\to \Vecc_K$ are isomorphic and the same follows for the two functors $\ba\to\Vecc_K$.
The other direction in part (1) follows from Remark~\ref{RemDefSigma}(2).

For part (2) assume that two such functors have the same homological kernel $\gamma$. By Theorem~\ref{ThmUniTensor}(2) and~\ref{RemBraid}, this leads to two non-isomorphic symmetric tensor functors $\bT_\gamma\to\Vecc$. The latter is contradicted by work of Deligne, see \cite[Theorem~6.4.1]{Tann} and \cite[Corollaire~6.20]{Del90}.\end{proof}


Recall from \cite[\S 8]{Del90} that to a symmetric monoidal functor $\theta:\ba\to\bT$ we can associate an affine group scheme $\underline{\mathrm{Aut}}^{\otimes}(\theta)$ in $\bT$. This is a representable functor from the category of commutative algebras in $\Ind\bT$ to the category of groups which sends an algebra $R$ to the automorphism group of the monoidal functor from $\ba$ to the category of $R$-modules in $\Ind\bT$ obtained by composing $\theta$ with extension of scalars. As an example of this we have the fundamental group $\pi(\bT)$ of $\bT$. We refer to \cite{Del90} for details.

\begin{theorem}
Assume that $k$ is perfect and that $\ba$ is symmetric and admits a faithful symmetric monoidal functor $\theta:\ba\to\bT$ of homological kernel $\gamma$ to an artinian symmetric tensor category $\bT$ over $k$. Consider the affine group scheme $G:=\underline{\mathrm{Aut}}^{\otimes}(\theta)$ in $\bT$ equipped with the natural transformation $\epsilon:\pi(\bT)\rightarrow G$ sending $\eta:\id_{\bT}\Rightarrow\id_{\bT}$ to $\eta_{\theta}$. Then there is an equivalence of symmetric tensor categories
$$\bT_\gamma\;\simeq\;\Rep_{\bT}(G,\epsilon).$$
\end{theorem}
\begin{proof}
By construction of $\bT_\gamma$, there exists a tensor functor $\omega:\bT_\gamma\to\bT$ such that $\theta\simeq\omega\circ \phi$, for the universal functor $\phi:\ba\to\bT_\gamma$. By \cite[\S8]{Del90}, we have an equivalence between $\bT_\gamma$ and $\Rep_{\bT}(G',\epsilon')$. Here $G'$ is the affine group scheme $\underline{\mathrm{Aut}}^{\otimes}(\omega)$ and $\epsilon':\pi(\bT)\to G'$ the corresponding homomorphism. It thus suffices to show that for every commutative ind-algebra $R$ in $\bT$, composition with $\phi$ yields an isomorphism
$$\Aut^{\otimes}(R\otimes-\circ \omega)\;\to\;\Aut^\otimes(R\otimes-\circ \theta).$$
The latter follows since the functors
$$\ABt^\otimes(\Ind\bT_\gamma,(\Ind\bT)_R)\;\to\; [\Noy\ba,(\Ind\bT)_R]^\otimes_{prex}\;\to\; [\ba,(\Ind\bT)_R]^\otimes$$
are fully faithful by Proposition~\ref{PropMon1} and Lemma~\ref{LemPrexMon}(1).
\end{proof}


\section{Examples}\label{Examples}
For the entire section, let $\ba$ be a {\em symmetric} rigid monoidal category over $k$. Tensor categories are again assumed to be over $k$ or some field extension. We will freely use notation and results from Appendix~\ref{App}.
\subsection{Prexact monoidal functors}
Throughout this subsection we assume that $k$ is {\em algebraically closed}.

\begin{theorem}\label{ThmVecFlat}
Assume that $\charr(k)=0$ and algebraically closed. Every faithful symmetric monoidal functor $\ba\to\Vecc_k$ is prexact.
\end{theorem}

Before preparing the proof of the theorem, we note the following consequence, which we obtain by combining the theorem with Corollary~\ref{CorFlatAll}.
\begin{corollary}\label{Corzero}
If $\charr(k)=0$ and there exists a faithful symmetric monoidal functor $\ba\to\Vecc_k$, then $\ba$ admits a weak abelian envelope.
\end{corollary}

\begin{lemma}\label{LemFullFlat}
Consider an affine group scheme $G/k$. Every fully faithful symmetric monoidal functor $\ba\to\Rep_k G$ is prexact.
\end{lemma}
\begin{proof}
By \cite[Theorem~4.2.1(1) and (2)]{CEOP}, there exists an affine group scheme $H/k$ and a fully faithful monoidal functor $\ba\to\Rep H$ such that every object in $\Rep H$ is a quotient of one in the image. That $\ba\to\Rep H$ is prexact thus follows from Example~\ref{TrivEx}. That the original functor $\ba\to\Rep G$ is also prexact follows from Corollary~\ref{CorFlatAll}.
\end{proof}

\begin{lemma}\label{UniFlat}
Assume that $\charr(k)=0$ and $n\in\mN$. For any $\Delta\in k^{\mN}$ and any symmetric monoidal functor $\phi:\EN(\Delta)\to\Vecc_k$, the faithful functor $\EN(\Delta)/{\ker\phi}\to\Vecc_k$ is prexact.
\end{lemma}
\begin{proof}
By Lemma~\ref{UniEN}, choosing the functor $\phi$ corresponds to choosing an appropriate endomorphism $A$ of a vector space $V$. Denote by $G{\hspace{-0.7mm}}_A$ the centraliser of $A$ in $GL(V)$. Clearly~$\phi$ lifts through the forgetful functor $\Rep G{\hspace{-0.7mm}}_A \to\Vecc$. It follows from \cite[Theorem~6.1]{KP} that the latter functor is full. Hence, by Lemma~\ref{LemFullFlat}, the functor $\EN(\Delta)/{\ker\phi}\to\Rep G{\hspace{-0.7mm}}_A$ is prexact, which implies that $\EN(\Delta)/{\ker\phi}\to\Vecc$ is prexact.
\end{proof}

\begin{proof}[Proof of Theorem~\ref{ThmVecFlat}]
Consider a functor $\theta:\ba\to\Vecc$ as in the theorem. For a given $f:Y=X\to X$ in $\ba$ we need to prove exactness of the sequence~\ref{ThmFlat}(b), by Remark~\ref{RemEnd}. We can consider the monoidal functor $\psi:\EN(\Delta)\to\ba$, corresponding to this endomorphism $f\in\ba(X,X)$ under the equivalence in Lemma~\ref{UniEN}. Set $n=\delta_0=\dim X$. By Lemma~\ref{UniFlat}, the composite
$$\EN(\Delta)/\ker{\psi}\;\to\; \ba\;\xrightarrow{\theta}\;\Vecc$$
is prexact. We denote the equivalence class of a morphism $m$ in $\EN(\Delta)$ in the quotient by $[m]$. By Theorem~\ref{ThmFlat} we find
$$\bigoplus_{h:B\to \bullet, \;[\varepsilon]\circ [h]=0}\theta(\psi(B))\;\to\; \theta(X)\;\to \;\theta(X)$$
is exact. In particular, we can rewrite the exact sequence as
$$\bigoplus_{h:B\to \bullet, \; f\circ \psi(h)=0}\theta(\psi(B))\;\to\; \theta(X)\;\to \;\theta(X).$$
This shows that also the sequence~\ref{ThmFlat}(b) must be exact, which concludes the proof.
\end{proof}

\begin{remark}
The above results are also obtained in much greater generality by O'Sullivan in \cite[Section~10]{OS}. We reformulate his results in the language of \cite{CEOP}. Let $k$ be a field of characteristic zero and $\ba$ such that every object in $\ba$ is annihilated by some Schur functor and $\cU(\ba)=\cU^{ep}(\ba)$, then $\Sh(\ba,\cT_{\cU^{ep}})$ is the ind-completion of a (super tannakian) tensor category. Hence, by the above and \cite[Corollary~2.3.3(1)]{CEOP}, Theorem~\ref{ThmVecFlat} (and Corollary~\ref{Corzero}) remains true if we replace $\Vecc_k$ by $\mathsf{sVec}_K$ for some field extension $K/k$
\end{remark}

\begin{prop}\label{PropOB} 
For every $n\in\mN$ and every symmetric monoidal functor $\theta:\OB_k(n)\to\Vecc_k$, the faithful functor $\OB_k(n)/{\ker\theta}\to\Vecc_k$ is prexact.
\end{prop}
\begin{proof}
It follows from the first fundamental theorem of invariant theory, see for instance \cite[Theorem~1.2]{KSX}, that the corresponding functor $\OB(n)\to\Rep GL(n)$ is full. That $\OB(n)/{\ker\theta}\to\Rep GL(n)$, and therefore also $\OB(n)/{\ker\theta}\to\Vecc$, is prexact thus follows from Lemma~\ref{LemFullFlat}.
\end{proof}
\begin{remark}
Alternatively, it follows also easily from \cite{PreTop, CEH, CEOP} that $\Rep GL(n)$ is the abelian envelope of the quotient in Proposition~\ref{PropOB}.
\end{remark}

\subsection{The universal case in prime characteristic}
We assume that char$(k)=p>0$ and show that $\OB_k(\delta)$ admits monoidal functors of many different homological kernels.

Recall from \cite{Tann} that for a symmetric tensor category $\bT$ over $k$ and $X\in\bT$, we define $\Fr_+X$ as the image of the composite
$$\Gamma^pX=H^0(S_p,\otimes^p X)\hookrightarrow \otimes^p X\tto H_0(S_p,\otimes^p X)=\Sym^p X.$$

We will also use the $p$-adic dimension $\Dim_+$ from \cite{EHO}. For $X\in\bT$, the $p$-adict integer $\Dim_+ X\in\mZ_{p}$ contains the information of all  categorical dimensions $\dim\Sym^jX$, $j\in\mN$, and $\dim X$ is the image of $\Dim_+X$ under $\mZ_p\tto\mF_p$.


\begin{theorem}\label{Thmp}
Set $\ba:=\OB_k(\delta)$ for some $\delta\in\mF_p\subset k$. Assume $p>2$.
\begin{enumerate}
\item There is a surjective function
$$\Sw(\ba)\;\tto\;\mZ_p\times\{0,1\},$$ 
which sends $\gamma$ to the following pair. With $F_\gamma:\ba\to\bT_\gamma$ the universal functor of homological kernel $\gamma$, the element of $\mZ_p$ is the image of $\Dim_+F_\gamma(V_\delta)$ under $\mZ_p\tto\mZ_p$, $\sum_{i\ge 0}a_ip^i\mapsto \sum_{i>0}a_ip^{i-1}$. For the second factor, we take $0$ when $\Fr_+(F_\gamma(V_\delta))=0$ and $1$ otherwise.
\item If $k$ is finite or algebraically closed, the above restricts to a surjective function 
$$\Sw_0(\ba)\tto\mZ_p\times \{0,1\}.$$
\end{enumerate}
For $p=2$ the statements remain true if we just consider $\Sw(\ba)\to\mZ_2$.
\end{theorem}

\begin{remark}
For $\delta\in k\backslash \mF_p$ we have $\Sw(\OB(\delta))=\varnothing$, by \cite[Lemma~2.2]{EHO}.
\end{remark}

The remainder of this section is devoted to the proof of the theorem. 
\subsubsection{}\label{SecU} We refer to \cite{Harman} for details on all of the constructions in this paragraph. Let $\cU$ be a non-principal ultrafilter on $\mN$. We have the ultrapower $K:=k^{\cU}$ which is a field extension of~$k$. If $k$ is a finite field, then $ k=K$ and if $k$ is algebraically closed then $K\simeq k$ non-canonically by Steinitz' theorem. 

Since the set of $p$-adic integers $\mZ_p$ is compact for the $p$-adic topology, we can define the ultralimit $\lim_{\cU}a_n$ of any sequence $\{a_n\in\mZ_p\,|\, n\in\mN\}$. This is $a\in\mZ_p$ for which, for every $l\in\mN$, the set of $n\in\mN$ for which $p^l| (a-a_n)$ is in $\cU$. We set $t(\cU)=\lim_{\cU}n\in\mZ_p$ and let $d(\cU)\in\mF_p\subset K$ denote the class of $(0,1,2,\cdots)\in k^{\mN}$ in $K=k^{\mN}/\sim$. This is the same as the image of $t(\cU)$ under $\mZ_p\tto \mF_p$. 

We define the ultrapower  $\bV(\cU):=\prod_{\cU}\Vecc_k$, which is a (non-artinian) symmetric tensor category over $K$. We let $W\in\bV(\cU)$ be the object corresponding to $(k^n)_n\in\Vecc^{\times\mN}_k$. In particular $\dim W=d(\cU)$.
As observed in \cite[\S 3.1]{EHO}, we have $\Dim_+W=t(\cU)$ and by construction $\Fr_+W\simeq W$.

We also define the ultrapower of the category of supervector spaces  $\bV'(\cU):=\prod_{\cU}\mathsf{sVec}_k$. We let $W'\in\bV'(\cU)$ be the object corresponding to $(k^{0|n})_n\in\mathsf{sVec}^{\times\mN}_k$. In particular $\dim W'=-d(\cU)$ and $\Dim_+W=-t(\cU)$ and by construction $\Fr_+W=0$.

\begin{proof}[Proof of Theorem~\ref{Thmp}]
Since $\Fr_+$ and $\Dim_+$ are preserved by symmetric tensor functors, by Theorem~\ref{ThmUniTensor} it suffices to show that $\OB_k(\delta)$ admits faithful symmetric monoidal functors to tensor categories which send $V_\delta$ to objects with the appropriate properties.

For part (1), by the discussion in \ref{SecU}, it thus suffices to observe the following two facts. Firstly, via the axiom of choice, for any $t\in\mZ_p$ we can choose an ultrafilter $\cU$ with $t(\cU)=t$. Secondly, for $d=d(\cU)$, the symmetric monoidal functors $\OB_k(d)\to \bV(\cU)$, $V_d\mapsto W$  and $\OB_k(-d)\to \bV'(\cU)$, $V_{-d}\mapsto W'$ are faithful. The latter follows from the observation that $kS_n\to \End_{k}(\otimes^n V)$ is injective if $V=k^{m}$ or $V=k^{0|m}$ with $m\ge n$. 
\end{proof}


\subsubsection{} It seems difficult to show directly that the homological kernel of a faithful symmetric monoidal functor $F$ from $\OB_k(\delta)$ determines the $p$-adic dimension of $F(V_\delta)$. On the other hand, we can easily show that it determines whether $\Fr_+(F(V_\delta))$ vanishes as follows. 

Consider the following morphism in $\ba:=\OB_k(\delta)$:
$$f:(\otimes^p V_\delta^\ast )\otimes (\otimes^p V_\delta)\;\to\; ((\otimes^p V_\delta^\ast )\otimes (\otimes^p V_\delta))^{2p-2} $$
where the $2p-2$ endomorphisms of $(\otimes^p V_\delta^\ast )\otimes (\otimes^p V_\delta)$ which make up $f$ are given by $(\id-s_i)\otimes \id$ and $\id\otimes (\id- s_i)$, with $\{s_i\,|\,1\le i<p\}$ the generators of $S_p$ acting via the symmetric braiding.


\begin{prop}\label{PropFr}
For $f$ as above, we have $\Sigma^{\ba}(f)\simeq k$. For a faithful symmetric monoidal functor $\theta:\ba\to\bT$, $V_\delta\mapsto X$ to a symmetric tensor category $\bT$, we have
$$
\Sigma^{\ba}\theta(f)\;=\; \begin{cases}
0&\mbox{ if } \Fr_+(X)\not=0,\\
k&\mbox{ if } \Fr_+(X)=0.\\
\end{cases}
$$
\end{prop}
\begin{proof}
By adjunction, the homomorphism $\ba(f,\unit)$
in the definition of $\Sigma^{\ba}(f)$ can be rewritten as $(kS_p)^{\times 2(p-1)} \to kS_p$, where the $2(p-1)$ maps $kS_p\to kS_p$ are given by left and right multiplication with $1-s_i$. The latter has one-dimensional cokernel, represented by the identity morphism, so $\dim_k\Sigma^{\ba}(f)\le 1$. 

We have $\ker\theta(f)=\Gamma^p (X^\ast)\otimes \Gamma^p X$. By the above (and $\Sigma^{\ba}\theta(f)\subset\Sigma^{\ba}(f)$) it now follows that $\Sigma^{\ba}\theta(f)\not=0$ if and only if the composition
$$ \Gamma^p (X^\ast)\otimes \Gamma^p X\;\hookrightarrow\; \otimes^p X^\ast\otimes\otimes^p X\;\xrightarrow{\ev_{\otimes^p X}}\; \unit,$$
is zero. The latter is equivalent with $\Fr_+X=0$. \end{proof}
%
%
%
%
%

\begin{remark}In \cite[Theorem~2.3.1]{AbEnv} it is proved that for a field $F$ of characteristic zero and $d\in F$, the category $\OB_F(d)$ is self-splitting. By Examples~\ref{ExSelf} and~\ref{ExPrexHK}, the same is thus {\bf not} true for $\OB_k(\delta)$ with $\delta\in \mF_p\subset k$.
\end{remark}

\subsection{More strictly subcanonical homological kernels}\label{SecMono}

\begin{lemma}
Consider a faithful monoidal functor $\theta:\bb\to \bT$ from a rigid monoidal category $\bb$ over $k$ to a tensor category $\bT$. If there is a monomorphism $\iota$ in $\bb$ which is not sent to a monomorphism in $\bT$, then the homological kernel of $\theta$ is not canonical.
\end{lemma}
\begin{proof}
Let $m:X\to Y$ denote the image of $\iota:A\to B$ under $\theta$. Since $f:=A^\ast\otimes \iota$ is also a monomorphism, by Example~\ref{ExNotMax}(2) it suffices to show that there exists a morphism $\alpha:A^\ast\otimes A\to\unit$ in $\ba$ for which $\theta(\alpha)$ does not restrict to zero on $\ker(X^\ast\otimes m)$. We take $\alpha=\ev_A$, which $\theta$ sends to $\ev_X$. By adjunction, the relevant restriction is zero if and only if $\ker(m)\to X$ (the inclusion of the kernel) is zero. This proves the lemma.
\end{proof}

We give concrete examples of the above situation.

\begin{example}
In the following two cases, the monomorphism $\mu$ from Lemma~\ref{mumono} is not sent to a monomorphism.
\begin{enumerate}
\item Assume that char$(k)=0$ and take $\delta\not\in\mZ$. The symmetric monoidal functor $\theta:\MO(2\delta,2\delta)\to \OB^1(\delta)$, to the semisimple tensor category $\OB^1(\delta)$, which sends $\mu$ to $\id_{\bullet}\oplus 0: \bullet\oplus\bullet \to \bullet\oplus\bullet$ is faithful.
\item Let $k$ be arbitrary and let $\cU$ be a non-principal ultrafilter on $\mN$. We use the notation of \ref{SecU}. For $\delta\in K=k^{\cU}$ given by the class of $(n)_n\in k^{\times \mN}$, the symmetric monoidal functor $\MO_k(\delta,\delta-1)\to \bV(\cU)$ which sends $\mu$ to a morphism represented by some surjections $\{k^n\tto k^{n-1}|n\in\mN\}$ is faithful.
\end{enumerate}
\end{example}

\begin{example}
The monomorphism $\iota$ in $\Seq$ from Lemma~\ref{Lemiota} is never sent to a monomorphism under any faithful monoidal functor to a tensor category (as it is sent to a morphism of the form $(a,b):A\oplus B\to C$ for an epimorphism $b:B\to C$ and a non-zero morphism $a:A\to C$).

Concrete examples of such faithful monoidal functors are given by the obvious functor $\Seq\to\OB^1(\delta)$, for  $\charr(k)=0$ and $\delta\not\in\mZ$, or can be constructed for arbitrary $k$ via ultrafilters.
\end{example}

\subsection{Motives}
Let $K$ be a field of characteristic zero and consider a Weil cohomology theory~$H^\bullet$, on the category of smooth projective varieties over a field $F$, with values in $K$.
Let $\Mot_H(F)$ denote the category of Chow motives modulo homological equivalence. This is a rigid symmetric monoidal category over $k=\mQ$ such that $H^\bullet$ extends to a faithful $\mQ$-linear symmetric monoidal functor $H:\Mot_H(F)\to\Vecc^{\mZ}_K$. 
In \cite[\S 3.2]{SchMot}, Sch\"appi constructed a tensor category $\mathbf{M}_H(F)$ yielding a commutative diagram
$$
\xymatrix{
\Mot_H(F)\ar[rr]\ar[rd]_H&&\mathbf{M}_H(F)\ar[ld]\\
&\Vecc^{\mZ}_K
}
$$
of symmetric monoidal functors where the right downwards arrow is a tensor functor.

\begin{lemma}
The tensor category $\mathbf{M}_H(F)$ is the universal tensor category associated to $\Mot_H(F)$ and the homological kernel of $H:\Mot_H(F)\to\Vecc_K^{\mZ}$ in Theorem~\ref{ThmUniTensor}. 
\end{lemma}
\begin{proof}Set $\ba:=\Mot_H(F)$.
Tracing through the construction of $\mathbf{M}_H(F)$ in \cite[\S 2.2 and \S 3.1]{SchMot}, we can conclude the following. In the notation of the current paper, $\mathbf{M}_H(F)$ is the smallest full subcategory of compact objects in $\Sh(K^b\ba,\cT)$, for $\cT$ the Grothendieck topology associated to $H_{\Delta}^{\mZ}:K^b\ba\to (\Vecc^{\mZ}_K)^{\mZ}$, which contains the image of $\ba$ and is closed under finite direct sums, (co)kernels, tensor products and duals. 

By Lemma~\ref{LemMinTensor} below, we can leave out the last two requirements. That $\mathbf{M}_H(F)$ is equivalent to the category of compact objects in $\Sh(K^b\ba,\cR)$, with $\cR$ the Grothendieck topology of $H_{\Delta}^{0}:K^b\ba\to \Vecc^{\mZ}_K$ then follows quickly from Theorem~\ref{ThmAlternative}(2).
\end{proof}

\begin{remark}
The new construction of $\mathbf{M}_H(F)$ as an instance of Theorem~\ref{ThmUniTensor} reveals the following:
\begin{enumerate}
\item $\mathbf{M}_H(F)$ has a universal property which extends the existence of $\mathbf{M}_H(F)\to\Vecc^{\mZ}_K$ from \cite{SchMot}.
\item We can describe (the ind-completion of) $\mathbf{M}_H(F)$ directly as a category of sheaves on $\Noy \Mot_H(F)$ or $K^b\Mot_H(F)$, rather than as a certain maximal subcategory of a category of sheaves on $K^b\Mot_H(F)$.
\end{enumerate}
\end{remark}

\begin{remark} The observations from this section extend to any `basic tannakian factorisation' as in \cite[\S 3.1]{SchMot}. 
\end{remark}

\begin{lemma}\label{LemMinTensor}
Consider a monoidal functor $\theta:\bb\to\bT$ from a rigid monoidal category $\bb$ to a tensor category $\bT$. Let $\bA\subset\bT$ be the minimal abelian subcategory that contains the image of $\theta$. Then $\bA$ is a tensor category, with tensor product inherited from $\bT$.
\end{lemma}
\begin{proof}
For some $B\in\bb$, we define $\bA_0\subset\bA$ the full subcategory of objects $X\in\bA$ for which $X\otimes\theta( B)\in\bA$. This is an abelian subcategory of $\bT$ which contains the image of $\theta$ and so $\bA_0=\bA$. This way we can prove that $\bA$ is closed under tensor product with objects in the image of $\theta$. Subsequently we can prove similarly that $\bA$ is a monoidal subcategory of $\bT$.
Finally, that $\bA$ is closed under taking duals follows from considering the full subcategory of all (left or right) duals of the objects in $\bA$. It is an abelian subcategory of $\bT$ containing the image of $\theta$, so it contains $\bA$, and it follows easily that this must in fact be an equality.
\end{proof}

\appendix

\section{Some universal monoidal categories}\label{App}

The {\em oriented Brauer category} $\OB$ from \cite{BCNR} was introduced as the universal rigid symmetric monoidal category on one object in \cite{Deligne}. We pretend that $\OB$ stands for `object', and introduce the universal rigid symmetric monoidal categories on one morphism and on one endomorphism as $\MO$ and $\EN$. Consider a field extension $K/k$ (potentially $K=k$).

For any category $\bC$ and a class of objects $C\subset\Ob\bC$, the {\bf groupoid corresponding to~$C$} is the subcategory of $\bC$ with objects $C$ and morphisms given by the isomorphisms in $\bC$.

\subsection{One object}\label{SecOB} Fix $\delta\in k$.
We denote by $\OB^0_k(\delta)$, the specialisation of the oriented Brauer category $\OB$ of \cite[\S 1]{BCNR} at $\delta$. Concretely, objects in $\OB^0_k(\delta)$ are words in the alphabet $\{\bullet,\circ\}$, and the morphism space between two such words is spanned by all pairings of the letters such that we only pair the same colour on different lines and different colours on the same line, when we place the target on a horizontal line above the source. An example of a morphism from $\bullet\circ\bullet\circ\bullet\circ$ to $\bullet\circ\circ\bullet\bullet\circ\circ\bullet$ is
$$
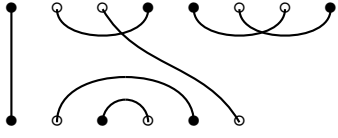
\begin{tikzpicture}[scale=1,thick,>=angle 90]
\begin{scope}[xshift=4cm]
\draw  (2.8,-0.5) -- +(0,1.5);
\draw (4,-0.5) to [out=90, in=180] +(0.3,0.3);
\draw (4.6,-0.5) to [out=90, in=0] +(-0.3,0.3);

\draw (3.4,-0.5) to [out=90, in=180] +(0.9,0.6);
\draw (5.2,-0.5) to [out=90, in=0] +(-0.9,0.6);

\draw (3.4,1) to [out=-90, in=-90] +(1.2,0);

\draw (5.2,1) to [out=-90, in=-90] +(1.2,0);

\draw (5.8,-0.5) to [out=120, in=-60] +(-1.8,1.5);

\draw (5.8,1) to [out=-90, in=-90] +(1.2,0);

\draw (3.4,-0.5) node[]{$\circ$};
\draw (4.6,-0.5) node[]{$\circ$};
\draw (5.8,-0.5) node[]{$\circ$};

\draw (3.4,1) node[]{$\circ$};
\draw (4,1) node[]{$\circ$};
\draw (5.8,1) node[]{$\circ$};
\draw (6.4,1) node[]{$\circ$};

\draw (2.8,-0.5) node[]{$\bullet$};
\draw (4,-0.5) node[]{$\bullet$};
\draw (5.2,-0.5) node[]{$\bullet$};

\draw (2.8,1) node[]{$\bullet$};
\draw (4.6,1) node[]{$\bullet$};
\draw (5.2,1) node[]{$\bullet$};
\draw (7,1) node[]{$\bullet$};
\end{scope}
\end{tikzpicture}
$$

Composition of morphisms corresponds to concatenation of the diagrams with evaluation of loops at $\delta$. The tensor product corresponds to juxtaposition of diagrams. We refer to \cite{BCNR, Deligne} for a full description. We denote by  $\OB_k(\delta)$, or $\OB(\delta)$ if $k$ is clear from the context, the additive envelope of  $\OB^0_k(\delta)$. We will also typically denote $\bullet$ by $V_\delta$.

The following well-known lemma follows as in \cite[Proposition~10.3]{Deligne}.

\begin{lemma}\label{UniOB}
Let $\ba$ be a rigid symmetric monoidal category over $K$. Evaluation at $V_\delta=\bullet$ yields an equivalence between $[\OB_k(\delta),\ba]^{s\otimes}$ and the groupoid corresponding to the objects in $\ba$ of dimension $\delta$.\end{lemma}

\begin{lemma}\label{LemZ}\cite[Th\'eor\`eme~10.5]{Deligne}
Assume $\mathrm{char}(k)=0$ and $\delta\not\in\mZ$.
The Karoubi envelope (idempotent completion) $\OB^1(\delta)$ of $\OB(\delta)$ is a semisimple tensor category over $k$.
\end{lemma}

\subsection{One morphism}
Fix $\delta,t\in k$.
We denote by $\MO^0_k(\delta,t)$ the following strict rigid monoidal category. Objects in $\MO^0_k(\delta,t)$ are words in the alphabet $\{\bullet,\circ, \blacksquare,\square\}$, and the space of morphisms between two such words has basis given by all pairings of the letters such that:
\begin{itemize}
\item Any occurrence of a symbol can be paired with an occurrence of the same symbol on the other line.
\item We can pair $\bullet$ with $\circ$, and $\blacksquare$ with $\square$, if they are on the same line.
\item We can pair $\bullet$ on the lower line with $\blacksquare$ on the top line; and $\square$ on the lower line with $\circ$ on the top line.
\end{itemize}
An example of a diagram representing a morphism from $\bullet\circ\bullet\circ\bullet\square$ to $\bullet\circ\circ\bullet\bullet\square\circ\blacksquare$ is
$$
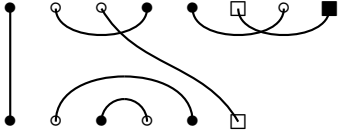
\begin{tikzpicture}[scale=1,thick,>=angle 90]
\begin{scope}[xshift=4cm]
\draw  (2.8,-0.5) -- +(0,1.5);
\draw (4,-0.5) to [out=90, in=180] +(0.3,0.3);
\draw (4.6,-0.5) to [out=90, in=0] +(-0.3,0.3);

\draw (3.4,-0.5) to [out=90, in=180] +(0.9,0.6);
\draw (5.2,-0.5) to [out=90, in=0] +(-0.9,0.6);

\draw (3.4,1) to [out=-90, in=-90] +(1.2,0);

\draw (5.2,1) to [out=-90, in=-90] +(1.2,0);

\draw (5.8,-0.5) to [out=120, in=-60] +(-1.8,1.5);

\draw (5.8,1) to [out=-90, in=-90] +(1.2,0);

\draw (3.4,-0.5) node[]{$\circ$};
\draw (4.6,-0.5) node[]{$\circ$};
\draw (5.8,-0.5) node[]{$\square$};

\draw (3.4,1) node[]{$\circ$};
\draw (4,1) node[]{$\circ$};
\draw (5.8,1) node[]{$\square$};
\draw (6.4,1) node[]{$\circ$};

\draw (2.8,-0.5) node[]{$\bullet$};
\draw (4,-0.5) node[]{$\bullet$};
\draw (5.2,-0.5) node[]{$\bullet$};

\draw (2.8,1) node[]{$\bullet$};
\draw (4.6,1) node[]{$\bullet$};
\draw (5.2,1) node[]{$\bullet$};
\draw (7,1) node[]{$\blacksquare$};
\end{scope}
\end{tikzpicture}
$$

Composition of morphisms corresponds to concatenation of the diagrams with evaluation of loops containing $\bullet$ (and hence also $\circ$) at $\delta$ and loops containing $\blacksquare$ at $t$. The tensor product corresponds to juxtaposition of diagrams. We denote by  $\MO_k(\delta,t)$, or $\MO(\delta,t)$ if $k$ is clear from the context, the additive envelope of  $\MO^0_k(\delta,t)$. We denote the morphism corresponding to the unique diagram from $\bullet$ to $\blacksquare$ by $\mu$.

\begin{lemma}\label{mumono}
The morphism $\mu:\bullet\to\blacksquare$ in $\MO(\delta,t)$ is a monomorphism.
\end{lemma}
\begin{proof}
It suffices to show that $\mu$ is a monomorphism in $\MO^0(\delta,t)$. For any word $w$ in $\{\bullet,\circ, \blacksquare,\square\}$, composition with $\mu$ yields an isomorphism between $\Hom(w,\bullet)$ and the subspace of $\Hom(w,\blacksquare)$ spanned by diagrams in which $\blacksquare$ on the top row is paired with some $\bullet$.
\end{proof}

\begin{lemma}\label{UniMO}
Let $\ba$ be a rigid symmetric monoidal category over $K$. Evaluation at $\mu$ yields an equivalence between $[\MO_k(\delta,t),\ba]^{s\otimes}$ and the groupoid corresponding to the objects in $\Arr\ba$ which are morphisms in $\ba$ from an object of dimension $\delta$ to one of dimension $t$.
\end{lemma}

\subsection{One endomorphism}
Fix  
$$\Delta=(\delta_0,\delta_1,\delta_2,\cdots)\in\prod_{i=0}^\infty k.$$ Denote by $\EN_k^0(\Delta)$ the specialisation of the associated graded of the affine oriented Brauer category (see \cite[\S 3.4]{BCNR}), where we specialise a loop with $i$ dots to $\delta_i$. Concretely, objects in $\EN_k^0(\Delta)$ are finite words in the alphabet $\{\bullet,\circ\}$ and the morphism space between two such words has as basis all `decorated' pairings, where the rules for the pairings are as in \ref{SecOB}, and each pairing is decorated with a natural number $i\in \mN$. We can depict this graphically by drawing an appropriate number of indicators on each strand in the diagram. Composition of morphisms corresponds to concatenation of the diagrams, where the number of indicators on a strand are simply added, with evaluation of loops with $i$ indicators at $\delta_i$. In this category we then have for instance
$$\End(\otimes^r \bullet)\;\simeq\; k[x_1,x_2, \cdots, x_r] \# S_r,$$
for the canonical action of the symmetric group $S_r$ on $k[x_1,x_2, \cdots, x_r]$. We denote the morphism $\bullet \to \bullet$ corresponding to a line with one indicator (`$x_1$' in the above equation) by~$\varepsilon$.

  The tensor product corresponds to juxtaposition of diagrams. We refer to \cite{BCNR} for a full description. We denote by  $\EN_k(\Delta)$, or $\EN(\Delta)$ if $k$ is clear from the context, the additive envelope of $\EN^0_k(\Delta)$.

The following is observed in \cite[Corollary~3.6]{BCNR}.
\begin{lemma}\label{UniEN}
Let $\ba$ be a rigid symmetric monoidal category over $K$. Evaluation at $\varepsilon$ yields an equivalence between $[\EN_k(\Delta),\ba]^{s\otimes}$ and the groupoid corresponding to the objects in $\Arr\ba$ which are endomorphisms $f$ in $\ba$ that satisfy $\mathrm{Tr}(f^i)=\delta_i$ for $i\in\mN$.
\end{lemma}

\subsection{One object, no braiding}\label{CatCSV}
We review the ($k$-linear version of the) universal (unbraided) rigid monoidal category from \cite{CSV}. Let $\Seq^0_k$ be the strict rigid monoidal category over $k$ where objects are words in the alphabet $\{X_i\,|\, i\in\mZ\}$. The morphism space between two words is spanned by all pairings which can be written diagrammatically without crossings and where each letter can be paired with itself on the other line; on the lower line $X_i$ can only be paired with some $X_{i+1}$ appearing to its left or $X_{i-1}$ on its right. On the upper line, $X_i$ can only be paired to $X_{i+1}$ appearing on its right or $X_{i-1}$ on its left. Composition and tensor product is again defined via concatenation and juxtaposition. We denote by $\Seq_k$ the additive envelope of $\Seq_k^0$.
Consider a morphism 
$$\iota:=(\co_{X_0},X_0\otimes\ev_{X_1}\otimes X_1):\;\unit \oplus X_0\otimes X_2\otimes X_1\otimes X_1\to X_0\otimes X_1.$$
The following lemma follows from the diagrammatic calculus.
\begin{lemma}\label{Lemiota}
The morphism $\iota$ is a monomorphism in $\Seq_k$.
\end{lemma}
The following lemma is proved as \cite[Proposition~6]{CSV}.
\begin{lemma}\label{UniSeq}
Let $\ba$ be a rigid monoidal category over $K$. Evaluation at $X_0$ yields an equivalence between $[\Seq_k,\ba]^{\otimes}$ and the groupoid corresponding to all objects in $\ba$.\end{lemma}

\subsection*{Acknowledgements}

The author thanks Richard Garner and Amnon Neeman for interesting discussions, Luca Barbieri Viale for introducing him to \cite{BHP} and Bruno Kahn for pointing out \cite{OS}. Some of the results in the current paper overlap with results such as \cite[1.14, 1.15, 3.8]{BHP}. The research was partially supported by ARC grant DP210100251.



\begin{thebibliography}
	{EGNO}


\bibitem[Be]{Be} A.~Beligiannis:
On the Freyd categories of an additive category. 
Homology Homotopy Appl. 2 (2000), 147--185. 


\bibitem[BEO]{BEO} D.~Benson, P.~Etingof, V.~Ostrik: New incompressible symmetric tensor categories in positive characteristic.  arXiv:2003.10499.

\bibitem[BHP]{BHP} L. Barbieri-Viale, A. Huber, M. Prest: Tensor structure for Nori motives.  Pacific J. Math. 306 (2020), no. 1, 1--30.



\bibitem[BQ]{BQ} F.~Borceux, C.~Quinteiro: A theory of enriched sheaves. Cahiers Topologie G\'eom. Diff\'erentielle Cat\'eg. 37 (1996), no. 2, 145--162.

\bibitem[BCNR]{BCNR} J.~Brundan, J.~Comes, D.~Nash, A.~Reynolds:
A basis theorem for the affine oriented Brauer category and its cyclotomic quotients. 
Quantum Topol. 8 (2017), no. 1, 75--112. 

\bibitem[Co1]{Tann}K. Coulembier: Tannakian categories in positive characteristic. Duke Math. J. 169 (2020), no. 16, 3167--3219.

\bibitem[Co2]{AbEnv} K. Coulembier: Monoidal abelian envelopes. Compos. Math. 157 (2021), no. 7, 1584--1609.

\bibitem[Co3]{PreTop}K. Coulembier: Additive Grothendieck pretopologies and presentations of tensor categories. arXiv:2011.02137.


\bibitem[CEH]{CEH} K. Coulembier, I. Entova-Aizenbud, T. Heidersdorf: Monoidal abelian envelopes and a conjecture of Benson - Etingof.  arXiv:1911.04303.

\bibitem[CEOP]{CEOP} K. Coulembier, P. Etingof, V. Ostrik, B. Pauwels: Monoidal abelian envelopes with a quotient property. arXiv:2103.00094.

\bibitem[CSV]{CSV} K. Coulembier, R. Street and M. van den Bergh: Freely adjoining monoidal duals. Math. Structures Comput. Sci. 31 (2021), no. 7, 748--768. 


	
	
	\bibitem[De1]{Del90} P.~Deligne: Cat\'egories tannakiennes. The Grothendieck Festschrift, Vol. II, 111--195, Progr. Math., 87, Birkh\"auser Boston, Boston, MA, 1990. 

	
\bibitem[De2]{Deligne}
P. Deligne:
La cat\'egorie des repr\'esentations du groupe sym\'etrique $S_t$, lorsque $t$ n'est pas un entier naturel.  Algebraic groups and homogeneous spaces, 209--273, 
{Tata Inst. Fund. Res. Stud. Math.}, Mumbai, 2007. 

\bibitem[Di1]{Di1} Y.~Diers: Familles universelles de morphismes. Ann. Soc. Sci. Bruxelles Sér. I 93 (1979), no. 3, 175--195


\bibitem[Di2]{Di2} Y.~Diers: Some spectra relative to functors. J. Pure Appl. Algebra 22 (1981), no. 1, 57--74.


\bibitem[EGNO]{EGNO}P.~Etingof, S.~Gelaki, D.~Nikshych, V.~Ostrik:
Tensor categories. 
Mathematical Surveys and Monographs, 205. American Mathematical Society, Providence, RI, 2015. 


\bibitem[EHO]{EHO}
P.~Etingof, N.~Harman, V.~Ostrik: p-adic dimensions in symmetric tensor categories in characteristic p. Quantum Topol. 9 (2018), no. 1, 119--140.



\bibitem[EHS]{EHS}
I.~Entova-Aizenbud, V.~Hinich, V.~Serganova:
Deligne categories and the limit of categories~$Rep(GL(m|n))$.
Int. Math. Res. Not. IMRN 2020, no. 15, 4602--4666.


\bibitem[Fr1]{Freyd1}
 P.~Freyd: Representations in abelian categories. 1966 Proc. Conf. Categorical Algebra, pp. 95–120 Springer, New York.


\bibitem[Fr2]{Freyd2}
P.~Freyd: Stable homotopy. 1966 Proc. Conf. Categorical Algebra, pp. 121–172 Springer, New York.

\bibitem[Gr]{Gro}
A.~Grothendieck: Sur quelques points d'algèbre homologique. Tohoku Math. J. (2) 9 (1957), 119--221.




\bibitem[Ha]{Harman} N.~Harman: Deligne categories as limits in rank and characteristic. arXiv:1601.03426.

\bibitem[KP]{KP}
H.~Kraft, C.~Procesi: Closures of conjugacy classes of matrices are normal. Invent. Math. 53 (1979), no. 3, 227--247.

\bibitem[KSX]{KSX}
 S.~K\"onig, I.H.~Slungård, C.C. Xi: Double centralizer properties, dominant dimension, and tilting modules. J. Algebra 240 (2001), no. 1, 393--412. 
 
 
\bibitem[Lo]{Lowen}
W.~Lowen: A generalization of the Gabriel-Popescu theorem. J. Pure Appl. Algebra 190 (2004), no. 1-3, 197--211.















\bibitem[Os]{Ostrik}
V.~Ostrik:
Module categories over representations of $SL_q(2)$ in the non-semisimple case.
Geom. Funct. Anal. 17 (2008), no. 6, 2005--2017.


\bibitem[OS]{OS}
P.~O'Sullivan:
Super Tannakian hulls. arXiv:2012.15703.


\bibitem[Sc]{SchMot}D.~Sch\"appi: Graded-Tannakian categories of motives. arXiv:2001.08567.


\bibitem[Ve]{Verdier}J.-L.~Verdier:
Des cat\'egories d\'eriv\'ees des cat\'egories ab\'eliennes.
Ast\'erisque No. 239 (1996).

 
 	\end{thebibliography}
\end{document}